\documentclass[11pt]{article}
\usepackage[T1]{fontenc}
\usepackage[utf8]{inputenc}
\usepackage{libertinus-type1}
\usepackage{microtype}
\usepackage[margin=1.15in]{geometry}
\usepackage{amsmath,amssymb,amsthm,mathtools,mathrsfs}
\usepackage{libertinust1math}
\usepackage{enumitem}
\usepackage{needspace}
\usepackage[dvipsnames]{xcolor}
\definecolor{paperlink}{HTML}{8B0000}
\definecolor{papercite}{HTML}{00CED1}
\definecolor{paperurl}{HTML}{EC008C}
\usepackage[nosort,nocompress]{cite}

\usepackage[colorlinks=true,linkcolor=paperlink,citecolor=papercite,urlcolor=paperurl]{hyperref}
\usepackage[nameinlink,noabbrev]{cleveref}
\numberwithin{equation}{section}
\newtheorem{theorem}{Theorem}[section]
\newtheorem{lemma}[theorem]{Lemma}

\newtheorem{proposition}[theorem]{Proposition}
\theoremstyle{remark}

\newcommand{\E}{\mathbb E}

\newcommand{\dd}{\,\mathrm d}

\allowdisplaybreaks[2]
\allowdisplaybreaks
\title{Large Deviation Principles for Area of Quantum Disks with Unit Left Boundary Length}
\author{
  Yuchen Fan\thanks{These authors contributed equally
  and are joint first authors.}\\
  {\small Courant Institute, New York University, USA}\\
  {\small yf2085@nyu.edu}
  \and
  Zhenfeng Tu\footnotemark[1]\\
  {\small Courant Institute, New York University, USA}\\
{\small zt2255@nyu.edu}
}
\date{}
\begin{document}
\maketitle
\begin{abstract}
We prove a large deviation principle for the area of a two-pointed
weight-$W$ quantum disk conditioned to have left boundary length one,
where $0<W<2$ is fixed and $\gamma\downarrow0$ without a subsequence
restriction. Conformal welding and the one-block
mating-of-trees law give independent Brownian blocks and a recursive
area decomposition. Growing moments and negative Laplace transforms
lead to two Hamilton--Jacobi equations, whose characteristic solutions
determine the limiting transforms. Direct comparison proves moment
convergence; compact-interval comparison and a change of measure prove
Laplace convergence. The resulting good rate function is specified by
scalar endpoint equations and is elementary when $W=1$.
\end{abstract}

\tableofcontents

\section{Introduction}\label{sec:overview}

Quantum disks are finite-area random surfaces in Liouville quantum
gravity. For $\gamma\in(0,2)$, they are described in conformal
coordinates by a random generalized function $h$, with quantum area
and boundary length formally written as $e^{\gamma h}\,d^2z$ and
$e^{\gamma h/2}\,ds$. Since $h$ is not pointwise defined, these
measures are constructed by regularization and renormalization,
using Gaussian multiplicative chaos; see \cite{RhodesVargas2014}
for background on this construction and its applications to Liouville
quantum gravity. Their conformal transformation
rules make them intrinsic measures on the quantum surface.

Quantum disks arise in the mating-of-trees framework of
Duplantier--Miller--Sheffield \cite{MoT}. We study the two-pointed
quantum disks of Ang--Holden--Sun
\cite{ConfWeldDisks}. Their two marked boundary points divide the
boundary into left and right arcs. The weight $W>0$ governs the
field singularities at these points and is additive under conformal
welding. For fixed $W>0$ and sufficiently small $\gamma$, the surface
has disk topology. Its natural law $\mathcal M(W)$ is an infinite
measure; fixing boundary lengths and normalizing the resulting
finite measure produces a probability law. Here we fix the left
length to one and integrate over the right length, leaving the
total area random. Section~\ref{sec:common-blocks} gives the precise
disintegration.

Our goal is to prove a large deviation principle (LDP) for quantum-disk
area as $\gamma\downarrow0$.
Throughout the paper, we fix $W\in(0,2)$. Let $A_\gamma$ be the total area of a two-pointed
weight-$W$ quantum disk conditioned to have left boundary length one,
with the right boundary length integrated out. We first prove the transform limits
under the integer-block condition
\[
 \theta=\frac{\pi\gamma^2}{4},\qquad
 N=\frac{2W}{\gamma^2}\in\mathbb N.
\]
This condition permits an exact decomposition into the conformal welding of $N$ disks of weight
$\gamma^2/2$. The one-block law from \cite{ConfWeldDisks} yields an
exact independent-block representation of $A_\gamma$ under this
condition. Section~\ref{sec:all-gamma} removes the restriction by
comparison with $\lfloor2W/\gamma^2\rfloor$ and
$\lfloor2W/\gamma^2\rfloor+1$ blocks at the same value of $\gamma$.
The main results hold as $\gamma\downarrow0$ through arbitrary values.

Our main result is an LDP for $A_\gamma$ with speed $\gamma^{-2}$.
Its proof rests on two transform limits: moments whose orders grow
like $\gamma^{-2}$, and the negative Laplace transform at the same
scale. Both transforms satisfy positive recursions obtained from a
single recursive decomposition of area. Their logarithmic limits
solve Hamilton--Jacobi equations, which we construct explicitly by
characteristics. Estimates on the discrete recursions then establish
convergence to these solutions.

The connection between logarithmic transforms, Hamilton--Jacobi
equations, and large deviations is developed systematically in
\cite{FengKurtz2006}. Here the Hamiltonians come from the exact
welding recursions, and we prove convergence directly for these
recursions. The final passage from transform limits to the LDP uses
the exponential-tilting argument underlying the G\"artner--Ellis
theorem \cite[Section~2.3]{DemboZeitouni1998}. We give that argument
explicitly, since the upper and lower tails involve different
observables, $\log A_\gamma$ and $A_\gamma$.

\subsection{Relation to FZZ formulas}

Ang--Remy--Sun--Zhu prove the FZZ boundary reflection formula in
\cite[Theorem~1.3]{BoundaryCFTConstants}. It encodes a regularized
joint transform of disk area and the two boundary lengths. Following notations in \cite{MoT},
\[
 Q=\frac2\gamma+\frac\gamma2,\qquad
 \beta=Q+\frac\gamma2-\frac W\gamma,
\]
the theorem's range $\beta\in(2/\gamma,Q)$ becomes
$W\in(\gamma^2/2,\gamma^2)$ and excludes fixed positive $W$ as
$\gamma\downarrow0$. This restriction does not, however, exclude our
regime from all the results of that paper: sufficiently high
derivatives in a boundary cosmological constant admit a probabilistic
representation for $\beta\in(\gamma/2,Q)$; see
\cite[Eq.~(3.59) and Lemma~B.5]{BoundaryCFTConstants}.
For fixed $0<W<2$, our parameters lie in this larger range for small
$\gamma$.

Our contribution is a direct route from positive welding recursions
to the growing-moment and negative-Laplace asymptotics needed for the
LDP. We do not extract these asymptotics from the analytically
continued reflection formula. Such an approach would require control
of its simultaneous parameter limits and of the conditioning on
left boundary length one; the formula alone does not supply that
asymptotic argument.

\subsection{Main results}

Define
\[
 \Lambda_{\gamma,W}(q)=\gamma^2\log\E[A_\gamma^{q/\gamma^2}],
 \quad 0<q<2,\qquad
 \Psi_{\gamma,W}(s)=\gamma^2\log\E[e^{-sA_\gamma/\gamma^2}],
 \quad s\ge0.
\]
The use of different transforms for the two tails reflects the lack
of positive exponential moments of $A_\gamma$ at fixed $\gamma$.
Set
\[
 a_*(W)=\frac{\pi W-\sin(\pi W)}{2\pi^2W^2}.
\]

\begin{theorem}[Limits of the two transforms]\label{thm:intro-transforms}
The transforms converge pointwise to finite limits $\Lambda_W$ and
$\Psi_W$. The function $\Lambda_W$ is real analytic on $(0,2)$ and
$\Psi_W$ is continuously differentiable on $(0,\infty)$.
Both limits extend continuously to value zero at the origin.
\end{theorem}

The limits are characterized by the Hamilton--Jacobi equations below.
The scalar formulas in Sections~\ref{sec:upper-results}
and~\ref{sec:lower-results} give the derivative ranges
\[
 \Lambda_W'((0,2))=(\log a_*(W),\infty),\qquad
 -\Psi_W'((0,\infty))=(0,a_*(W)).
\]
Consequently, exponential tilting by the two transforms covers all
positive areas other than $a_*(W)$. Define
\[
 I_W(a)=
 \begin{cases}
  \displaystyle\sup_{s\ge0}\{-sa-\Psi_W(s)\},&0<a<a_*(W),\\[1mm]
  0,&a=a_*(W),\\[1mm]
  \displaystyle\sup_{0\le q<2}\{q\log a-\Lambda_W(q)\},&a>a_*(W),\\[1mm]
  +\infty,&a\le0.
 \end{cases}
\]

\begin{theorem}[Area large deviations]\label{thm:area-ldp}\label{thm:main}
As $\gamma\downarrow0$ through arbitrary values, the variables $A_\gamma$
satisfy an LDP on $\mathbb R$ with speed $\gamma^{-2}$ and good rate
function $I_W$.
The rate vanishes only at $a_*(W)$; in particular,
$A_\gamma\to a_*(W)$ in probability.
\end{theorem}

Equivalently, for every Borel set $E\subset\mathbb R$,
\[
 -\inf_{a\in E^\circ}I_W(a)
 \le\liminf_{\gamma\downarrow0}\gamma^2\log\mathbb{P}(A_\gamma\in E)
 \le\limsup_{\gamma\downarrow0}\gamma^2\log\mathbb{P}(A_\gamma\in E)
 \le-\inf_{a\in\overline E}I_W(a).
\]
No local limit or density asymptotics are asserted.

When $W=1$, the rate function simplifies to
\[
 I_1(a)=2\log(2\pi a)-2+\frac1{\pi a},\qquad a>0.
\]
Theorem~\ref{thm:area-ldp} follows from the transform limits and
exponential tilting; its proof is given in Section~\ref{sec:full-ldp}.
The remainder of the introduction describes the common recursion,
the two limiting equations, and the estimates used to pass to the
limit.

\subsection{Recursive decomposition of area}

The welding theorem of \cite{ConfWeldDisks} and the one-block
mating-of-trees description
give independent pairs $(\ell_k,T_k)$, $1\le k\le N$, such that
\[
 A_\gamma\overset d=
 \frac{\sin\theta}{2}\sum_{k=1}^N
 \left(\prod_{r<k}\ell_r^2\right)T_k.
\]
Here $\ell_k$ is the ratio of the right to the left boundary length
of block $k$. Its density is proportional to
$\ell_k^{-(N-k)}\ell_k^{\pi/\theta-1}/(1+\ell_k^{\pi/\theta})^2$.
Conditionally on $\ell_k$, $T_k$ is the duration of a standard planar
Brownian excursion across a cone of angle $\theta$, from radius
$1/\sin\theta$ to radius $\ell_k/\sin\theta$.
The joint law, including its normalization, is derived in
Proposition~\ref{prop:block}.

For $1\le m\le N$, the normalized area of the last $m$ blocks is
\[
 A^{(m)}:=\frac{\sin\theta}{2}
 \sum_{k=N-m+1}^N
 \left(\prod_{r=N-m+1}^{k-1}\ell_r^2\right)T_k,
 \qquad A^{(0)}:=0.
\]
Here and below an empty product equals one. These blocks have incoming
boundary length $\prod_{r=1}^{N-m}\ell_r$ in the original disk;
$A^{(m)}$ is their area after rescaling this length to one. Since area
scales quadratically with boundary length, their contribution to the
original area is $(\prod_{r=1}^{N-m}\ell_r^2)A^{(m)}$.
In particular,
$A^{(1)}=(\sin\theta/2)T_N$ and, for $N\ge2$,
$A^{(2)}=(\sin\theta/2)(T_{N-1}+\ell_{N-1}^2T_N)$.
We henceforth use this realization of $A_\gamma$, so that
$A^{(N)}=A_\gamma$.

Separating the leftmost block gives the exact identity
\begin{equation}\label{intro:eq:area-recursion}
 A^{(m)}
 =\frac{\sin\theta}{2}T_{N-m+1}
  +\ell_{N-m+1}^2A^{(m-1)}.
\end{equation}
The new pair $(\ell_{N-m+1},T_{N-m+1})$ is independent of $A^{(m-1)}$
and has boundary bias $-(m-1)\theta/\pi$.
The factor $\ell_{N-m+1}^2$ accounts for the change of boundary-length
normalization of the remaining blocks. Identity~\eqref{intro:eq:area-recursion}
is the common starting point for the moment and Laplace recursions.

\subsection{Exact recursions for the transforms}

The two recursions are expressed in terms of
$Z(b)=\pi b/\sin(\pi b)$, with $Z(0)=1$, and
\[
 H(v)=
 \begin{cases}
  \pi\sqrt v/\sin(\pi\sqrt v),&0<v<1,\\
  1,&v=0,\\
  \pi\sqrt{-v}/\sinh(\pi\sqrt{-v}),&v<0.
 \end{cases}
\]
\paragraph{Growing moments.}
For integers $j\ge0$ with $j<2/\gamma^2$, define
\[
 Q_{m,j}
 =\frac{\prod_{i=0}^{m-1}Z(i\theta/\pi)}
 {j![\theta^2/(\pi^2\sin\theta)]^j}\E[(A^{(m)})^j],
 \qquad
 w_\theta(b,d)=H[b^2,(b+2\theta/\pi)^2,\ldots,(b+2d\theta/\pi)^2].
\]
Brackets denote divided differences, with continuous extension when
inputs coincide; their definition is recalled in
Section~\ref{sec:moment-exact}.
Expanding \eqref{intro:eq:area-recursion} by the binomial theorem and
evaluating the mixed block moments gives
\begin{equation}\label{intro:eq:moment-recursion}
 Q_{m,j}=\sum_{d=0}^j
 w_\theta\!\left(\frac{(2(j-d)-(m-1))\theta}{\pi},d\right)
 Q_{m-1,j-d},\qquad
 Q_{0,0}=1,\quad Q_{0,j}=0\ (j>0).
\end{equation}
In the binomial expansion, $d$ is the power of the new block's area
and $j-d$ the power of the remaining area. The normalization of
$Q_{m,j}$ cancels the factorials and deterministic factors in the
mixed block moments. The calculation is given in
Section~\ref{sec:moment-exact}.

\paragraph{Negative Laplace transforms.}
Fix $s\ge0$ and define
$f_m(x)=\E[e^{-s e^{2x}A^{(m)}/\gamma^2}]$.
The variable $x$ is the logarithm of the incoming boundary length, so
$e^{2x}A^{(m)}$ is the corresponding area.
Let $V_m$ be a sum of independent exponential variables with rates
$j^2-((m-1)\theta/\pi)^2$, $j\ge1$, independent of a standard
Brownian motion $B$, and put
\[
 X_v^{(m)}
 =x+\frac{\sqrt2\theta}{\pi}B_v
       -\frac{2(m-1)\theta^2}{\pi^2}v .
\]
The logarithm map sends the cone to a strip. Its conformal time change,
followed by the boundary tilt, turns \eqref{intro:eq:area-recursion} into
\begin{equation}\label{intro:eq:laplace-recursion}
 f_m(x)=(\mathcal K_m f_{m-1})(x)
 =\E\!\left[
 e^{-\frac{s\gamma^2}{16\sin\theta}\int_0^{V_m}e^{2X_v^{(m)}}\,\dd v}
 f_{m-1}(X_{V_m}^{(m)})\right],\qquad f_0=1.
\end{equation}
Here $f_N(0)=\E[e^{-sA_\gamma/\gamma^2}]$ and
\[
 \E[e^{\lambda V_m}]
 =\frac{H(((m-1)\theta/\pi)^2+\lambda)}
        {H(((m-1)\theta/\pi)^2)},
 \qquad \lambda<1-((m-1)\theta/\pi)^2.
\]
Independent Brownian motions and clocks are used in successive blocks.
Section~\ref{lower:sec:clock} proves these identities.
The positivity of both recursions permits comparison with exponential
test functions constructed from the limiting equations.

\subsection{Hamilton--Jacobi limits}

The derivation of the moment equation requires a further normalization
to control weights with large $d$. Define
\[
 V(b)=-\frac12\int_0^b\log(1-r^2)\,\dd r,\qquad
 \widetilde Q_{m,j}
 =e^{-\frac4{\gamma^2}V((2j-m)\theta/\pi)}Q_{m,j}.
\]
This replaces $w_\theta(b,d)$ by
$\widetilde w_\theta(b,d)
=w_\theta(b,d)e^{-\frac4{\gamma^2}[V(b+(2d-1)\theta/\pi)-V(b)]}$.
When all the quantities $b+2h\theta/\pi$, $0\le h\le d$, and
$b+(2d-1)\theta/\pi$ remain in a fixed compact subinterval of $(-1,1)$,
the normalized weights are bounded above and away from zero uniformly
in $d$. This includes $d$ of order $\gamma^{-2}$.
For fixed $d$, their limits and generating function are
\[
 a_d(b)=(1-b^2)^{d-1/2}\frac{H^{(d)}(b^2)}{d!},\qquad
 A(b,z):=\sum_{d\ge0}a_d(b)z^d
 =\frac{H(b^2+(1-b^2)z)}{\sqrt{1-b^2}}.
\]
Here $|b|<1$ and $0<z<1$. The uniform estimates are proved in
Section~\ref{upper:sec:moment-recursion}.

Set $t=m\theta/\pi$ and $u=j\theta/\pi$.
Formally inserting $\widetilde Q_{m,j}\approx e^{4G(t,u)/\gamma^2}$
into the recursion gives $G_t=\log A(2u-t,e^{-G_u})$, with $G_u>0$.
Equivalently, $F(t,u)=G(t,u)+V(2u-t)$ satisfies
\begin{equation}\label{intro:eq:moment-hj}
 F_t=\log H((2u-t)^2+e^{-F_u}),\qquad
 (2u-t)^2+e^{-F_u}<1.
\end{equation}
The characteristic solution extends continuously to the axes, with
$G(0,u)=0$ and $G(t,0)=\int_0^t\log a_0(-r)\,\dd r$.
The value at $t=0$ is a limiting boundary value; comparison with the
discrete recursion starts at its first positive row.

Similarly, the ansatz
$f_{m-1}(x)\approx e^{4\mathcal G((m-1)\theta/\pi,x)/\gamma^2}$
and the Gaussian and clock transforms give
\begin{equation}\label{intro:eq:laplace-hj}
 \mathcal G_t
 =\log H\!\left((\mathcal G_x-t)^2-\frac{s}{4\pi}e^{2x}\right)
  -\log H(t^2),\qquad \mathcal G(0,x)=0.
\end{equation}
The required transform domain is
$(\mathcal G_x-t)^2-\frac{s}{4\pi}e^{2x}<1$.

Both equations admit an explicit characteristic construction.
For the moment equation, set $b=2u-t$, $z=e^{-F_u}$,
$v=b^2+z$, and $h=H'/H$. Along a characteristic,
\[
 u'=zh(v),\qquad z'=-4bzh(v),\qquad v'=-2b.
\]
For the Laplace equation, put $p=\mathcal G_x-t$ and
$v=p^2-\frac{s}{4\pi}e^{2x}$. Its characteristic equations are
\[
 x'=-2ph(v),\qquad
 p'=-1-\frac{s}{2\pi}e^{2x}h(v),\qquad v'=-2p.
\]
Define $\chi(v)=\cos(\pi\sqrt v)$ for $0\le v<1$ and
$\chi(v)=\cosh(\pi\sqrt{-v})$ for $v<0$.
In both constructions,
\[
 \frac{\dd^2}{\dd t^2}\chi(v(t))+\pi^2\chi(v(t))=0,
 \qquad \chi(v(t))=C\cos(\pi t).
\]
The characteristic families are parametrized by $-1<C<1$ for moments
and $C>1$ for the negative Laplace transform, with $v<1$ throughout.
The endpoint maps are strictly monotone, so each endpoint in the
relevant domain determines a unique characteristic. The characteristic
action then defines the solutions and verifies their strict transform
domains. These constructions and their boundary behavior are proved in
Sections~\ref{upper:sec:moment-recursion} and~\ref{lower:sec:hj}.

In terms of these solutions, the limiting transforms are
\[
 \Lambda_W(q)
 =q\log\frac{q}{4\pi}-q+4F(W/2,q/4)
       -4\int_0^{W/2}\log Z(t)\,\dd t,\qquad
 \Psi_W(s)=4\mathcal G(W/2,0),
\]
where the dependence of $\mathcal G$ on $s$ is suppressed.

\subsection{Convergence of the recursions}

The formal expansions alone do not establish convergence: $d$ can be
of order $\gamma^{-2}$ in
\eqref{intro:eq:moment-recursion}, and the Brownian clock in
\eqref{intro:eq:laplace-recursion} has unbounded support.
Both arguments therefore require tail estimates before the expansion
can be justified.

\paragraph{Moment comparison.}
We use exponential test functions obtained by evaluating $G$ at
$(t+\eta,u+\eta)$, with $\eta>0$ fixed.
On the resulting compact rectangle, $G_u\ge p_\eta>0$ and $|G_t|\le M_\eta$.
The exact ratio of consecutive exponential test values is bounded by
$e^{M_\eta-p_\eta d}$. Uniform boundedness of $\widetilde w_\theta$
therefore controls the entire sum over $d>K$ by a geometric tail.
Taylor expansion is applied only to $d\le K$, first fixing $K$ and then
sending $\gamma\to0$. The HJ equation makes the limiting one-step ratio
equal to one, up to an error vanishing with $\eta$.
Direct estimates on the first positive row and the strip $u<\eta$
supply the boundary comparison. Positivity then propagates upper and
lower bounds through all $N$ steps. An error $e^{\pm\varepsilon}$ per
step contributes only $\pm(N\gamma^2/4)\varepsilon$ to the scaled
logarithm. Taking $\gamma\downarrow0$ and then $\eta\downarrow0$
proves convergence; see Section~\ref{sec:convergence}.

\paragraph{Laplace comparison and change of measure.}
For the upper bound, we first restrict the spatial process to $(-M,M)$.
The strict inequality for the argument of $H$ gives exponential control
of large clocks. For bounded clocks, a change of measure applied to
$e^{4\mathcal G/\gamma^2}$ gives the one-step upper comparison.
An exponential estimate for paths leaving $(-M,M)$ removes the spatial
restriction after iteration. These estimates control the weighted
expectation, as required for a logarithmic asymptotic.

For the lower bound, choose the characteristic $x_*(t)$ ending at
$x_*(W/2)=0$. Set $t_m=(m-1)\theta/\pi$,
$P(t)=\mathcal G_x(t,x_*(t))$, and
$v_*(t)=(P(t)-t)^2-\frac{s}{4\pi}e^{2x_*(t)}$.
The log-length increment of block $N-m+1$ is
$\Delta_m=\log\ell_{N-m+1}
=(\sqrt2\theta/\pi)B_{V_m}-2(m-1)\theta^2V_m/\pi^2$.
We change the law of each block by the density
\[
 L_m=
 \exp\!\left\{\frac{4P(t_m)}{\gamma^2}\Delta_m
       -\frac{s}{4\pi}e^{2x_*(t_m)}V_m\right\}
 \frac{H(t_m^2)}{H(v_*(t_m))}.
\]
The Gaussian and clock transforms give $\E L_m=1$ exactly.
Under the product change of measure, the mean increments follow the
characteristic backwards from $t=W/2$ to $t=0$, in the order
$m=N,N-1,\ldots,1$ of the physical traversal.
The sum of the increment variances is $O(\gamma^2)$, and clock moment
bounds also control the interiors of the blocks. Consequently,
\[
 \frac{\gamma^2}{4}\log\!\left[
 e^{-sA_\gamma/\gamma^2}\prod_{m=1}^N L_m^{-1}\right]
 \longrightarrow\mathcal G(W/2,0)
\]
in probability under the changed law. Restricting the expectation to
an event of probability tending to one on which this logarithm is
close to its limit gives the lower bound. Both Laplace bounds are
proved in Section~\ref{lower:sec:convergence}.

\paragraph{From transforms to large deviations.}
Power tilting concentrates $\log A_\gamma$ near
$\Lambda_W'(q)$, while negative exponential tilting concentrates
$A_\gamma$ near $-\Psi_W'(s)$. Their slope ranges cover the upper and
lower deviations, respectively. Markov bounds and the reverse
change-of-measure bounds give the tail rates; exponential tightness and
fixed-neighborhood bounds yield Theorem~\ref{thm:area-ldp}.
Section~\ref{sec:full-ldp} formulates the common tilting argument and
completes the proof.

To remove the integer-block condition, we weld a residual disk on the
right, first to obtain the weight-$W$ disk from $N$ ordinary blocks,
and then to obtain $N+1$ ordinary blocks from the weight-$W$ disk.
Area additivity gives the two comparisons. The marginal law of the
left component is reweighted by a power of its right boundary length
with exponent in $[-1,0]$. We prove uniform bounds on the moments of
these densities and their inverses, so H\"older's inequality removes
the reweighting at the large-deviation scale. Both integer-block
endpoints converge to $W/2$, and hence give the same transform limits.
This argument is carried out in Section~\ref{sec:all-gamma} before the
final passage to the LDP.

\subsection{Reading guide}

Section~\ref{sec:common-blocks} fixes the quantum-disk normalization
and derives the common independent-block representation. The proof
then follows two routes. For upper deviations,
Section~\ref{sec:moment-exact} derives the exact moment recursion,
Section~\ref{upper:sec:moment-recursion} constructs its
Hamilton--Jacobi limit, and Section~\ref{sec:convergence} proves
convergence to that limit. For lower deviations,
Section~\ref{lower:sec:clock} derives the exact Laplace recursion,
Section~\ref{lower:sec:hj} constructs its Hamilton--Jacobi limit,
and Section~\ref{lower:sec:convergence} proves convergence using
comparison and a change of measure. The two routes meet in
Section~\ref{sec:all-gamma}, which removes the integer-block
restriction, and Section~\ref{sec:full-ldp}, which converts the
transform limits into the full LDP and evaluates the rate explicitly
when $W=1$.

Readers interested first in the resulting rate function can read
Section~\ref{sec:common-blocks}, the endpoint formulas in
Subsections~\ref{sec:upper-results} and~\ref{sec:lower-results},
and then Section~\ref{sec:full-ldp}. The intervening recursion and
convergence sections justify those formulas; Section~\ref{sec:all-gamma}
explains why they hold along arbitrary sequences $\gamma\downarrow0$.

\section{Quantum disks and the block representation}\label{sec:common-blocks}

Until Section~\ref{sec:all-gamma}, we impose
$N=2W/\gamma^2\in\mathbb N$ and $\theta=\pi\gamma^2/4$.

We use the two-pointed quantum-disk measure $\mathcal M(W)$ and its
continuous boundary-length disintegration
$\mathcal M(W;\ell,r)$ from \cite{ConfWeldDisks}.
The two marked points divide the boundary into left and right arcs.
Our probability law is the normalization of
$\mathcal M(W;1,\bullet):=\int_0^\infty\mathcal M(W;1,r)\,\dd r$.
This measure has finite positive mass: the left-boundary marginal is
a finite positive constant times
$\ell^{-2W/\gamma^2}\,\dd\ell$ when
$\gamma^2/2\le W<2+\gamma^2$
\cite[Lemma~2.16]{ConfWeldDisks}. Our sequence
$W=N\gamma^2/2<2$ lies in this range.

Adding $(2/\gamma)\log a$ to the field multiplies quantum boundary
lengths by $a$ and quantum area by $a^2$. We use this scaling below.
The field construction itself is not needed in the subsequent analysis;
the required geometric inputs are the welding identity and the
one-block area law.

\subsection{Geometric inputs and normalization}

For positive lengths $r_0,\ldots,r_N$, welding $N$ weight-$\gamma^2/2$
disks produces a weight-$W$ disk. After forgetting the welding curves,
the measure identity is
\begin{equation}\label{eq:welding-input}
 \mathcal M(W;r_0,r_N)
 =c_{\gamma,N}\int_{(0,\infty)^{N-1}}
 \operatorname{Weld}_*
 \left[\bigotimes_{k=1}^N
 \mathcal M(\gamma^2/2;r_{k-1},r_k)\right]
 \prod_{k=1}^{N-1}\dd r_k.
\end{equation}
Here $\operatorname{Weld}_*$ denotes the image under conformal welding,
and $c_{\gamma,N}>0$ is independent of the lengths.
This is \cite[Theorem~2.3]{ConfWeldDisks}; for $N=1$ no welding is
necessary. Conditional on the lengths, the component disks are
independent and their quantum areas add.

Write
\[
 K(x,y)=\frac{(xy)^{\pi/\theta-1}}
 {(x^{\pi/\theta}+y^{\pi/\theta})^2}.
\]
The mass of $\mathcal M(\gamma^2/2;x,y)$ is proportional to $K(x,y)$.
Its normalized area law is $(\sin\theta/2)$ times the duration of
a Brownian excursion in the quadrant, from $(x,0)$ to $(0,y)$, with
coordinate variances $t$ and covariance $-t\cos\theta$.
This follows from \cite[Proposition~7.7]{ConfWeldDisks} and the
mating-of-trees variance $2/\sin\theta$ identified in \cite{FZZWelding};
see also \cite[Remark~7.2]{ConfWeldDisks}.
In particular, converting the Brownian clock with variance rate
$2/\sin\theta$ to one with variance rate one gives the area multiplier
$\sin\theta/2$.

The linear map
\[
 (x,y)\longmapsto
 \left(\frac{x+y\cos\theta}{\sin\theta},y\right)
\]
takes this variance-one correlated Brownian motion to standard planar
Brownian motion, without changing time. It maps the quadrant to a cone
of angle $\theta$, and the endpoints to radii
$x/\sin\theta$ and $y/\sin\theta$ on its two sides.
This explains the exact radius in the cone kernel; $1/\theta$ would
only be an asymptotic approximation.

Use the following normalization of the cone excursion kernel:
\[
 \mathsf K_\theta(a,b,t)
 =\frac{e^{-(a^2+b^2)/(2t)}}{abt}
 \sum_{j\ge1}(-1)^{j+1}j^2 I_{j\pi/\theta}(ab/t).
\]
Here $I_\nu$ is the modified Bessel function of the first kind.
Taking the two inward normal derivatives of the killed cone heat
kernel gives a constant multiple of this kernel. That constant cancels
when conditioning on the endpoints. Thus the conditional duration
density is proportional to $\mathsf K_\theta(a,b,t)\,\dd t$.
Its integrated form is
\begin{equation}\label{eq:integrated-cone-kernel}
 \int_0^\infty
 \mathsf K_\theta\!\left(\frac1{\sin\theta},
                       \frac{l}{\sin\theta},t\right)\,\dd t
 =\frac{\theta\sin^2\theta}{\pi}
   \frac{l^{\pi/\theta-1}}{(1+l^{\pi/\theta})^2}.
\end{equation}
For example, the normalized exit-radius density is obtained by mapping
the cone to the upper half-plane with $z\mapsto z^{\pi/\theta}$;
the entrance limit of the Poisson kernel gives
$(\pi/\theta)l^{\pi/\theta-1}/(1+l^{\pi/\theta})^2$.
The displayed normalization is also obtained by integrating the Bessel
expansion. Only the normalized conditional density is needed for the
geometric identification.

For $1\le k\le N$, define a pair $(\ell_k,T_k)$ by the density
\begin{equation}\label{eq:block-pair-density}
 \frac{\pi^2}
 {\theta^2\sin^2\theta\,Z(-(N-k)\theta/\pi)}
 l^{-(N-k)}
 \mathsf K_\theta\!\left(\frac1{\sin\theta},
                       \frac{l}{\sin\theta},t\right)\,\dd l\,\dd t,
 \qquad l,t>0.
\end{equation}
Take these pairs independently. The following proposition verifies
normalization and identifies the quantum area; it replaces any
independent block-law assumption.

\begin{proposition}[Exact quantum-disk block representation]
\label{prop:block}
The densities \eqref{eq:block-pair-density} are probability densities,
and the area of the normalized weight-$W$ disk with left length one
satisfies
\begin{equation}\label{eq:block-area}
 A_\gamma\overset d=
 \frac{\sin\theta}{2}\sum_{k=1}^N
 \left(\prod_{j<k}\ell_j^2\right)T_k.
\end{equation}
\end{proposition}

\begin{proof}
Set $r_0=1$ and integrate \eqref{eq:welding-input} over the free final
length $r_N$. The joint boundary-length density is proportional to
$\prod_{k=1}^N K(r_{k-1},r_k)\prod_{k=1}^N\dd r_k$.
Put $\ell_k=r_k/r_{k-1}$, including the final ratio $\ell_N$.
The transformation $r_k=\prod_{j\le k}\ell_j$ has Jacobian
\[
 \prod_{k=1}^N r_{k-1}=\prod_{j=1}^N\ell_j^{N-j}.
\]
Since $K(cx,cy)=c^{-2}K(x,y)$, the kernel product contributes
$\prod_{j=1}^N\ell_j^{-2(N-j)}\prod_{k=1}^N K(1,\ell_k)$.
Combining the two factors gives
\[
 \prod_{k=1}^N
 \frac{\ell_k^{\pi/\theta-(N-k)-1}}
      {(1+\ell_k^{\pi/\theta})^2}\,\dd\ell_k.
\]
In particular, the boundary bias is negative.

To normalize one factor, substitute $y=l^{\pi/\theta}$ and use the beta
integral. For $|b|<1$,
\[
 \int_0^\infty
 \frac{l^{(\pi/\theta)(1+b)-1}}{(1+l^{\pi/\theta})^2}\,\dd l
 =\frac{\theta}{\pi}\Gamma(1+b)\Gamma(1-b)
 =\frac{\theta}{\pi}Z(b).
\]
Here $b=-(N-k)\theta/\pi$ lies in $(-1,0]$, since $N\theta/\pi=W/2<1$.
Thus all factors have finite positive mass, and the boundary ratios
are independent with densities
\begin{equation}\label{eq:block-ratio-density}
 \frac{\pi}{\theta Z(-(N-k)\theta/\pi)}
 \frac{l^{\pi/\theta-(N-k)-1}}{(1+l^{\pi/\theta})^2}\,\dd l.
\end{equation}

Conditional on these ratios, scale the $k$-th disk by $r_{k-1}^{-1}$
in boundary length. Its boundary lengths become $(1,\ell_k)$ and its
area is divided by $r_{k-1}^2$. The one-block area law identifies this
rescaled area as $(\sin\theta/2)T_k$, with conditional duration density
\[
 \frac{\mathsf K_\theta(1/\sin\theta,l/\sin\theta,t)}
 {\int_0^\infty\mathsf K_\theta(1/\sin\theta,l/\sin\theta,u)\,\dd u}
 \,\dd t.
\]
Multiplying this density by \eqref{eq:block-ratio-density} and using
\eqref{eq:integrated-cone-kernel} gives exactly
\eqref{eq:block-pair-density}. Conditional independence of the component
disks and factorization of the ratio density prove independence of the
pairs. Additivity of area under welding proves \eqref{eq:block-area}.
The final ratio $\ell_N$ is auxiliary: integrating it out gives the
terminal duration law, without any extra tilt.
\end{proof}

All remaining arguments use Proposition~\ref{prop:block}. The geometric
input is now a consequence of the cited welding and mating-of-trees
results, rather than a hypothesis about an unspecified area variable.

\subsection{Scalar functions used in both HJ equations}
Retain $Z,H,\chi$ from the introduction and set $h(v)=H'(v)/H(v)>0$.
The map $\chi:(-\infty,1)\to(-1,\infty)$ is strictly decreasing, with
$\chi'(v)=-\pi^2/(2H(v))$. The symbol $F$ below denotes the continuous
moment comparison function, not a cone kernel.

\subsection{The common area recursion}\label{sec:partial-areas}

We work on the probability space of the independent pairs in
Proposition~\ref{prop:block}, and realize $A_\gamma$ by
\eqref{eq:block-area}. The physical blocks are ordered $1,\ldots,N$.
Write $r_0=1$ and $r_k=\prod_{j=1}^k\ell_j$ for their successive
boundary lengths. The $k$-th block has boundary lengths
$(r_{k-1},r_k)$ and area $(\sin\theta/2)r_{k-1}^2T_k$.

Fix $1\le m\le N$. The last $m$ blocks have indices
$N-m+1,\ldots,N$, and their incoming boundary length is $r_{N-m}$.
Divide every boundary length in this group by $r_{N-m}$; this divides
every area by $r_{N-m}^2$. Define the resulting total area by
\begin{equation}\label{eq:partial-area-definition}
 A^{(m)}:=\frac{\sin\theta}{2}
 \sum_{k=N-m+1}^N
 \left(\prod_{r=N-m+1}^{k-1}\ell_r^2\right)T_k.
\end{equation}
Set $A^{(0)}=0$ and interpret empty products as one. Equivalently,
\[
 r_{N-m}^2A^{(m)}
 =\frac{\sin\theta}{2}\sum_{k=N-m+1}^N r_{k-1}^2T_k
\]
is the area contributed by these blocks before rescaling.
This rescaling is performed separately for each $m$; the variables
$A^{(m)}$ are not the unrescaled partial sums of the total area.

For example, whenever the indicated blocks exist,
\[
 \begin{aligned}
 A^{(1)}&=\frac{\sin\theta}{2}T_N,\\
 A^{(2)}&=\frac{\sin\theta}{2}
       (T_{N-1}+\ell_{N-1}^2T_N),\\
 A^{(3)}&=\frac{\sin\theta}{2}
       (T_{N-2}+\ell_{N-2}^2T_{N-1}
                  +\ell_{N-2}^2\ell_{N-1}^2T_N).
 \end{aligned}
\]
Thus increasing $m$ means adding a block on the left, not appending
one on the right. The variable $A^{(m-1)}$ depends only on the pairs
with indices $N-m+2,\ldots,N$, and hence is independent of the new
pair $(\ell_{N-m+1},T_{N-m+1})$.

\begin{lemma}[Adding one block]\label{lem:partial-area-recursion}
For $1\le m\le N$,
\[
 A^{(m)}=\frac{\sin\theta}{2}T_{N-m+1}
       +\ell_{N-m+1}^2A^{(m-1)},
 \qquad A^{(N)}=A_\gamma.
\]
The new block is independent of $A^{(m-1)}$.
\end{lemma}

\begin{proof}
In \eqref{eq:partial-area-definition}, the term $k=N-m+1$ is
$(\sin\theta/2)T_{N-m+1}$. For every $k\ge N-m+2$,
\[
 \prod_{r=N-m+1}^{k-1}\ell_r^2
 =\ell_{N-m+1}^2\prod_{r=N-m+2}^{k-1}\ell_r^2.
\]
The remaining sum is therefore $\ell_{N-m+1}^2A^{(m-1)}$.
This factor is necessary because the last $m-1$ blocks have incoming
length $\ell_{N-m+1}$ in the coordinates used for $A^{(m)}$, whereas
their incoming length is one in the definition of $A^{(m-1)}$.
For $m=1$ the remaining sum is empty and $A^{(0)}=0$.
Independence was noted above, and $m=N$ gives \eqref{eq:block-area}.
Finally, substituting $k=N-m+1$ into the block bias
$-(N-k)\theta/\pi$ gives $-(m-1)\theta/\pi$.
\end{proof}

\section{Exact moment recursion}\label{sec:moment-exact}

Fix $0<q<2$ and integers $n$ with $\gamma^2n\to q$.
We first compute the mixed block moments and then apply the common area
recursion to obtain $Q_{m,j}$.

\subsection{Mixed block moments}

We derive the mixed moments of the boundary ratio and the Brownian
running time.  These moments will be used when expanding the total
area as a sum of block contributions.

For \(b\in(-1,1)\), let \(\mathbb E_b\) denote expectation under the
one-block density
\[
 \frac{\pi^2}{\theta^2\sin^2\theta\,Z(b)}
 l^{\pi b/\theta}
 \mathsf K_\theta
 \left(\frac1{\sin\theta},\frac{l}{\sin\theta},t\right)
 \dd l\,\dd t.
\]
For block \(k\) in Proposition~\ref{prop:block}, the bias is
\(b=-(N-k)\theta/\pi\).

We first compute an integral that appears after expanding the cone
kernel into its Bessel series.

For the next lemma, fix an integer \(d\ge0\), a number \(\nu>0\),
and \(\xi\) satisfying \(-\nu<\xi<\nu-2d\).

\begin{lemma}[An integrated Bessel term]
\label{lem:integrated-bessel-term}
Under these conditions,
\begin{align}
 &\int_0^\infty\int_0^\infty
 l^{\xi-1}u^{d-1}
 e^{-(1+l^2)/(2u)}I_\nu(l/u)\,\dd u\,\dd l
 \notag\\
 &\qquad=
 \frac{2^{d+1}d!}
 {\displaystyle\prod_{h=0}^d
 \bigl[\nu^2-(\xi+2h)^2\bigr]}.
 \label{eq:integrated-bessel-term}
\end{align}
\end{lemma}

\begin{proof}
The modified Bessel function has the positive series
\[
 I_\nu(l/u)
 =
 \sum_{m=0}^\infty
 \frac{(l/(2u))^{\nu+2m}}
 {m!\,\Gamma(\nu+m+1)}.
\]
Thus Tonelli's theorem allows us to integrate term by term.
The assumptions imply \(\nu>d\).  Substituting
\(r=(1+l^2)/(2u)\), we obtain
\[
 \int_0^\infty
 u^{d-\nu-2m-1}e^{-(1+l^2)/(2u)}\,\dd u
 =
 \Gamma(\nu+2m-d)
 \left(\frac{1+l^2}{2}\right)^{d-\nu-2m}.
\]
Next, the beta integral gives
\begin{align*}
 &\int_0^\infty
 l^{\xi+\nu+2m-1}(1+l^2)^{d-\nu-2m}\,\dd l\\
 &\qquad=
 \frac{
 \Gamma\left(\frac{\nu+\xi}{2}+m\right)
 \Gamma\left(\frac{\nu-\xi}{2}+m-d\right)}
 {2\Gamma(\nu+2m-d)}.
\end{align*}
The gamma factors in the numerator have positive arguments by the
assumptions on \(\xi\).

Combining these two integrations, the left-hand side of
\eqref{eq:integrated-bessel-term} becomes
\[
 2^{-d-1}
 \sum_{m=0}^\infty
 \frac{
 \Gamma\left(\frac{\nu+\xi}{2}+m\right)
 \Gamma\left(\frac{\nu-\xi}{2}-d+m\right)}
 {m!\,\Gamma(\nu+1+m)}.
\]
We evaluate this series using the identity
\[
 \sum_{m=0}^\infty
 \frac{\Gamma(a+m)\Gamma(b+m)}
 {m!\,\Gamma(c+m)}
 =
 \frac{\Gamma(a)\Gamma(b)\Gamma(c-a-b)}
 {\Gamma(c-a)\Gamma(c-b)},
 \qquad
 a,b>0,\quad c>a+b.
\]
For completeness, this identity follows by writing
\[
 \frac{\Gamma(b+m)}{\Gamma(c+m)}
 =
 \frac1{\Gamma(c-b)}
 \int_0^1 r^{b+m-1}(1-r)^{c-b-1}\,\dd r
\]
and summing
\(\sum_{m\geq0}\Gamma(a+m)r^m/m!
=\Gamma(a)(1-r)^{-a}\).
The remaining integral is another beta integral.

Applying the identity with
\(a=(\nu+\xi)/2\), \(b=(\nu-\xi)/2-d\), and \(c=\nu+1\)
gives
\[
 2^{-d-1}d!\,
 \frac{
 \Gamma\left(\frac{\nu+\xi}{2}\right)
 \Gamma\left(\frac{\nu-\xi}{2}-d\right)}
 {
 \Gamma\left(\frac{\nu-\xi}{2}+1\right)
 \Gamma\left(\frac{\nu+\xi}{2}+d+1\right)}.
\]
Finally,
\[
 \frac{\Gamma(a)}{\Gamma(a+d+1)}
 =
 \frac1{\prod_{h=0}^d(a+h)}.
\]
Applying this to the two gamma ratios and pairing the factors yields
\[
 \prod_{h=0}^d
 \left(\frac{\nu+\xi}{2}+h\right)
 \left(\frac{\nu-\xi}{2}-h\right)
 =
 4^{-(d+1)}
 \prod_{h=0}^d\bigl[\nu^2-(\xi+2h)^2\bigr].
\]
This proves \eqref{eq:integrated-bessel-term}.
\end{proof}

For mixed moments, fix \(b\in(-1,1)\), an integer \(d\ge0\),
and \(p\in\mathbb R\) satisfying
\[
 b+\frac{p\theta}{\pi}>-1,
 \qquad
 b+\frac{(p+2d)\theta}{\pi}<1.
\]
For a function \(J\), brackets denote divided differences:
\(J[v_0]=J(v_0)\) and
\[
 J[v_0,\ldots,v_d]
 =\frac{J[v_1,\ldots,v_d]-J[v_0,\ldots,v_{d-1}]}{v_d-v_0},
\]
with continuous extension when inputs coincide. In particular,
\(J[v,\ldots,v]=J^{(d)}(v)/d!\) for \(d+1\) equal inputs.

\begin{lemma}[Mixed block moments]\label{lem:mixed-block-moments}
The mixed moment is finite and equals
\begin{equation}\label{eq:mixed}
 \mathbb E_b[\ell^pT^d]
 =
 \frac{d!}{Z(b)}
 \left(\frac{2\theta^2}{\pi^2\sin^2\theta}\right)^d
 H\left[
 \left(b+\frac{p\theta}{\pi}\right)^2,
 \left(b+\frac{(p+2)\theta}{\pi}\right)^2,
 \ldots,
 \left(b+\frac{(p+2d)\theta}{\pi}\right)^2
 \right].
\end{equation}
\end{lemma}

\begin{proof}
We first treat \(d=0\).  Integrating out the running time gives the
boundary-ratio density
\[
 \frac{\pi}{\theta Z(b)}
 \frac{l^{\pi(1+b)/\theta-1}}
 {(1+l^{\pi/\theta})^2}\,\dd l.
\]
With the substitution \(y=l^{\pi/\theta}\), we obtain
\begin{align*}
 \mathbb E_b[\ell^p]
 &=
 \frac1{Z(b)}
 \int_0^\infty
 \frac{y^{b+p\theta/\pi}}{(1+y)^2}\,\dd y\\
 &=
 \frac{
 \Gamma(1+b+p\theta/\pi)
 \Gamma(1-b-p\theta/\pi)}
 {Z(b)}\\
 &=
 \frac{Z(b+p\theta/\pi)}{Z(b)}.
\end{align*}
The assumptions ensure that both beta-integral parameters are positive.
Since \(Z(x)=H(x^2)\), this is \eqref{eq:mixed} for \(d=0\).

Now suppose \(d\geq1\).  To remove the initial Brownian radius,
substitute \(t=u/\sin^2\theta\).  Directly from the definition of the
kernel,
\[
 \mathsf K_\theta
 \left(\frac1{\sin\theta},
       \frac{l}{\sin\theta},
       \frac{u}{\sin^2\theta}\right)
 =
 \sin^4\theta\,\mathsf K_\theta(1,l,u).
\]
Therefore
\begin{align}
 \mathbb E_b[\ell^pT^d]
 &=
 \frac{\pi^2}{\theta^2Z(b)\sin^{2d}\theta}
 \int_0^\infty\int_0^\infty
 l^{\pi b/\theta+p}u^d
 \mathsf K_\theta(1,l,u)\,\dd u\,\dd l.
 \label{eq:mixed-rescaling}
\end{align}
Expanding the kernel gives
\begin{align*}
 &l^{\pi b/\theta+p}u^d\mathsf K_\theta(1,l,u)\\
 &\qquad=
 \sum_{j=1}^\infty(-1)^{j+1}j^2
 l^{\pi b/\theta+p-1}u^{d-1}
 e^{-(1+l^2)/(2u)}I_{j\pi/\theta}(l/u).
\end{align*}
Lemma~\ref{lem:integrated-bessel-term}, with
\(\nu=j\pi/\theta\) and \(\xi=\pi b/\theta+p\), evaluates the
integral of each Bessel term as
\[
 \frac{2^{d+1}d!}
 {\displaystyle\prod_{h=0}^d
 \left[
 \left(\frac{j\pi}{\theta}\right)^2
 -
 \left(\frac{\pi b}{\theta}+p+2h\right)^2
 \right]}.
\]
The hypotheses of that lemma hold for every \(j\geq1\).
Moreover, after multiplication by \(j^2\), these positive integrals
are \(O(j^{-2d})\) as \(j\to\infty\).
Since \(d\geq1\), their sum is finite.  Thus the Bessel series may be
integrated term by term despite its alternating signs.

Substituting into \eqref{eq:mixed-rescaling} and extracting
\((\pi/\theta)^2\) from each denominator factor gives
\begin{align}
 \mathbb E_b[\ell^pT^d]
 =
 \frac{2^{d+1}d!}{Z(b)}
 \left(\frac{\theta^2}{\pi^2\sin^2\theta}\right)^d
 \sum_{j=1}^\infty
 \frac{(-1)^{j+1}j^2}
 {\displaystyle\prod_{h=0}^d
 \left[
 j^2-
 \left(b+\frac{(p+2h)\theta}{\pi}\right)^2
 \right]}.
 \label{eq:mixed-series}
\end{align}

It remains to identify this series as a divided difference of \(H\).
The partial-fraction expansion of
\(H(v)=\pi\sqrt v/\sin(\pi\sqrt v)\) is
\[
 H(v)
 =
 1+2v\sum_{j=1}^\infty
 \frac{(-1)^{j+1}}{j^2-v},
 \qquad 0\leq v<1.
\]
This series converges locally uniformly.  For the resolvent
\(v\mapsto(j^2-v)^{-1}\), its divided difference is
\[
 \left[\frac1{j^2-v}\right][v_0,\ldots,v_d]
 =
 \prod_{h=0}^d\frac1{j^2-v_h}.
\]
For \(d\geq1\), the constant term has zero divided difference, and
\(v/(j^2-v)=j^2/(j^2-v)-1\).  Consequently,
\[
 H[v_0,\ldots,v_d]
 =
 2\sum_{j=1}^\infty
 \frac{(-1)^{j+1}j^2}
 {\prod_{h=0}^d(j^2-v_h)}.
\]
The identity extends to coincident nodes by continuity.
Taking
\(v_h=(b+(p+2h)\theta/\pi)^2\)
in this formula identifies the series in
\eqref{eq:mixed-series} and proves \eqref{eq:mixed}.
\end{proof}

\begin{lemma}[Application to the area expansion]
\label{lem:mixed-moments-area-expansion}
For every integer \(0\leq n<2/\gamma^2=\pi/(2\theta)\), the
multinomial expansion of \(\mathbb E[A_\gamma^n]\) involves only
finite mixed moments covered by
Lemma~\ref{lem:mixed-block-moments}.
\end{lemma}

\begin{proof}
Choose nonnegative integers \(d_1,\ldots,d_N\) with
\(\sum_{k=1}^Nd_k=n\).  The corresponding term in the expansion of
\[
 A_\gamma^n
 =
 \left[
 \frac{\sin\theta}{2}
 \sum_{k=1}^N
 \left(\prod_{i<k}\ell_i^2\right)T_k
 \right]^n
\]
is
\[
 n!\left(\frac{\sin\theta}{2}\right)^n
 \prod_{k=1}^N
 \frac{
 T_k^{d_k}\ell_k^{\,2\sum_{j>k}d_j}}
 {d_k!}.
\]
Indeed, \(\ell_k^2\) occurs in every area contribution from a block
with index larger than \(k\).

Independence of the block pairs therefore gives
\[
 \mathbb E[A_\gamma^n]
 =
 n!\left(\frac{\sin\theta}{2}\right)^n
 \sum_{\substack{d_1,\ldots,d_N\geq0\\
                 d_1+\cdots+d_N=n}}
 \prod_{k=1}^N
 \frac{
 \mathbb E_{-(N-k)\theta/\pi}
 \left[
 \ell^{\,2\sum_{j>k}d_j}T^{d_k}
 \right]}
 {d_k!}.
\]
For block \(k\), the smallest unsquared node in
\eqref{eq:mixed} is
\[
 -\frac{(N-k)\theta}{\pi}
 +\frac{2\theta}{\pi}\sum_{j>k}d_j
 \geq-\frac{(N-1)\theta}{\pi}>-1.
\]
The largest is
\[
 -\frac{(N-k)\theta}{\pi}
 +\frac{2\theta}{\pi}\sum_{j\geq k}d_j
 \leq\frac{2n\theta}{\pi}<1.
\]
Here we used \(N\theta/\pi=W/2<1\).
Thus every mixed moment is finite and satisfies
\eqref{eq:mixed}.

The factors \(d_k!\) in that formula cancel the multinomial
denominators.  The remaining deterministic factor is
\[
 \left(\frac{\sin\theta}{2}\right)^n
 \prod_{k=1}^N
 \left(\frac{2\theta^2}{\pi^2\sin^2\theta}\right)^{d_k}
 =
 \left(\frac{\theta^2}{\pi^2\sin\theta}\right)^n.
\]
Hence
\begin{align}
 \mathbb E[A_\gamma^n]
 ={}&
 \frac{
 n!\left(\frac{\theta^2}{\pi^2\sin\theta}\right)^n}
 {\displaystyle\prod_{k=1}^N
 Z\left(-\frac{(N-k)\theta}{\pi}\right)}
 \notag\\
 &\times
 \sum_{\substack{d_1,\ldots,d_N\geq0\\
                 d_1+\cdots+d_N=n}}
 \prod_{k=1}^N
 H\left[
 \left(
 \frac{\theta}{\pi}
 \left[-(N-k)+2\sum_{j>k}d_j+2h\right]
 \right)^2
 \,:\,h=0,\ldots,d_k
 \right].
 \label{eq:area-composition-moments}
\end{align}
In the last line, the notation inside the brackets lists the
\(d_k+1\) divided-difference nodes in order.
This is the positive sum reorganized into a recurrence below.
\end{proof}

\subsection{From the multinomial expansion to a recurrence}

Use the partial areas from Section~\ref{sec:partial-areas}. The index $m$
counts blocks and $j$ denotes moment order.

For $0\le j\le n$, define the normalized moments
\[
 Q_{m,j}:=
 \frac{\prod_{i=0}^{m-1}Z(i\theta/\pi)}
 {j!\left(\frac{\theta^2}{\pi^2\sin\theta}\right)^j}
 \E[(A^{(m)})^j].
\]
We use the empty-product convention and take every zeroth moment
to be one. In particular, $Q_{0,0}=1$ and $Q_{0,j}=0$ for $j>0$.
Since $q<2$, the target order satisfies $n<2/\gamma^2$ for all
sufficiently small $\gamma$, so all these moments are finite.
The normalization removes the factorials, the common scale factor,
and the boundary-bias normalizers from the following recurrence.

Write
\[
 w_\theta(b,d)
 =H\left[b^2,(b+2\theta/\pi)^2,\ldots,(b+2d\theta/\pi)^2\right],
\]
\begin{lemma}[Exact recurrence for the normalized moments]
For $1\le m\le N$ and $0\le j\le n$,
\begin{equation}\label{upper:eq:exact-recursion}
 Q_{m,j}
 =\sum_{d=0}^j
 w_\theta\left(\frac{(2(j-d)-(m-1))\theta}{\pi},d\right)
 Q_{m-1,j-d}.
\end{equation}
In particular,
\begin{equation}\label{upper:eq:exact-moment}
 \E[A_\gamma^n]
 =
 n!\left(\frac{\theta^2}{\pi^2\sin\theta}\right)^n
 \frac{Q_{N,n}}{\prod_{j=0}^{N-1}Z(j\theta/\pi)}.
\end{equation}
\end{lemma}

\begin{proof}
The binomial theorem and independence give
\[
 \E[(A^{(m)})^j]
 =\sum_{d=0}^j\binom jd\left(\frac{\sin\theta}{2}\right)^d
 \E_{-(m-1)\theta/\pi}[\ell^{2(j-d)}T^d]
 \E[(A^{(m-1)})^{j-d}].
\]
Here $d$ is the power assigned to the new block's area, and $j-d$
is the power assigned to the remaining area. By \eqref{eq:mixed},
\[
 \left(\frac{\sin\theta}{2}\right)^d
 \E_{-(m-1)\theta/\pi}[\ell^{2(j-d)}T^d]
 =\frac{d!\left(\frac{\theta^2}{\pi^2\sin\theta}\right)^d}
 {Z((m-1)\theta/\pi)}
 w_\theta\left(\frac{(2(j-d)-(m-1))\theta}{\pi},d\right).
\]
We used the evenness of $Z$. The argument of $w_\theta$ includes
both the boundary bias and the power $2(j-d)$ of $\ell$.
The definition of the previous row gives
\[
 \E[(A^{(m-1)})^{j-d}]
 =\frac{(j-d)!\left(\frac{\theta^2}{\pi^2\sin\theta}\right)^{j-d}}
 {\prod_{i=0}^{m-2}Z(i\theta/\pi)}Q_{m-1,j-d}.
\]
Substituting these identities into the binomial expansion, use
$\binom jd d!(j-d)!=j!$ and
$Z((m-1)\theta/\pi)\prod_{i=0}^{m-2}Z(i\theta/\pi)
=\prod_{i=0}^{m-1}Z(i\theta/\pi)$ to obtain
\[
 \E[(A^{(m)})^j]
 =\frac{j!\left(\frac{\theta^2}{\pi^2\sin\theta}\right)^j}
 {\prod_{i=0}^{m-1}Z(i\theta/\pi)}
 \sum_{d=0}^j
 w_\theta\left(\frac{(2(j-d)-(m-1))\theta}{\pi},d\right)
 Q_{m-1,j-d}.
\]
Dividing by the common prefactor proves
\eqref{upper:eq:exact-recursion}. Equation~\eqref{upper:eq:exact-moment}
is the definition of $Q_{N,n}$, since $A^{(N)}=A_\gamma$.
\end{proof}

The recurrence determines successive rows from the initial row
$Q_{0,j}$. All weights are positive, since the product
$H(v)=\prod_{k\geq1}(1-v/k^2)^{-1}$ has positive Taylor coefficients.
Thus only the exponential rate of $Q_{N,n}$ remains to be determined.
The index $j$ below ranges from zero to the target moment order $n$.

\section{Exact Laplace recursion}\label{lower:sec:clock}

Fix $s\ge0$. With the partial areas of Section~\ref{sec:partial-areas},
define
\begin{equation}\label{lower:eq:def-fn}
 f_m(x)=\E\!\left[e^{-s e^{2x}A^{(m)}/\gamma^2}\right],
 \qquad 0\le m\le N.
\end{equation}
Thus $f_0=1$ and
$\Psi_{\gamma,W}(s)=\gamma^2\log f_N(0)$.
Lemma~\ref{lem:partial-area-recursion} and independence give
\begin{equation}\label{lower:eq:preliminary-recursion}
 f_n(x)=\E\!\left[
 e^{-\frac{s\sin\theta}{2\gamma^2}e^{2x}T_{N+1-n}}
 f_{n-1}(x+\log\ell_{N+1-n})\right].
\end{equation}
This is the Laplace counterpart of the binomial moment recursion.
The remaining task is to express the same block law using a Brownian
motion and an independent clock.

We next rewrite the joint law of
\((\log\ell_{N+1-n},T_{N+1-n})\) in terms of a Brownian motion and an
independent random clock. Let \(V_0=\sum_{j=1}^{\infty}E_j\), where the \(E_j\)'s are independent
exponential random variables with rates \(j^2\). Since
\(\sum_{j\geq1}\E[E_j]=\sum_{j\geq1}j^{-2}<\infty\), the series defining
\(V_0\) is finite almost surely. Euler's product formula gives
\begin{equation}\label{lower:eq:unbiased-clock-mgf}
\E[e^{vV_0}]
=
\prod_{j=1}^{\infty}\frac{j^2}{j^2-v}
=
H(v),
\qquad v<1.
\end{equation}

For \(1\leq n\leq N\), define \(V_n\) by
\begin{equation}\label{lower:eq:tilted-clock-law}
\mathbb{P}(V_n\in\dd v)
=
\frac{
\exp\left\{\left(\frac{(n-1)\theta}{\pi}\right)^2v\right\}
}{
H\left(\left(\frac{(n-1)\theta}{\pi}\right)^2\right)
}
\mathbb{P}(V_0\in\dd v).
\end{equation}
Equivalently, \(V_n\) is a sum of independent exponential random
variables with rates
\[
j^2-\left(\frac{(n-1)\theta}{\pi}\right)^2,
\qquad j\geq1.
\]
The assumption \(N\theta<\pi\) ensures that all these rates are positive.
Moreover,
\begin{equation}\label{lower:eq:clockmgf}
\E[e^{\lambda V_n}]
=
\frac{
H\left(\left(\frac{(n-1)\theta}{\pi}\right)^2+\lambda\right)
}{
H\left(\left(\frac{(n-1)\theta}{\pi}\right)^2\right)
},
\qquad
\lambda<
1-\left(\frac{(n-1)\theta}{\pi}\right)^2.
\end{equation}

Let \(B\) be a standard Brownian motion independent of \(V_n\). For a
starting log-boundary length \(x\), define
\begin{equation}\label{lower:eq:block-log-process}
X_v^{(n)}
=
x+\frac{\sqrt2\,\theta}{\pi}B_v
-\frac{2(n-1)\theta^2}{\pi^2}v,
\qquad
0\leq v\leq V_n.
\end{equation}
For every nonnegative measurable function \(g\), define
\begin{equation}\label{lower:eq:positiveoperator}
(\mathcal K_ng)(x)
=
\E\left[
\exp\left\{
-\frac{s\gamma^2}{16\sin\theta}
\int_0^{V_n}e^{2X_v^{(n)}}\dd v
\right\}
g\left(X_{V_n}^{(n)}\right)
\right].
\end{equation}
The operator \(\mathcal K_n\) is positive and order preserving: if
\(0\leq g_1\leq g_2\), then
\[
0\leq\mathcal K_ng_1\leq\mathcal K_ng_2.
\]

\begin{proposition}[Exact Brownian-clock recursion]
\label{lower:prop:exact-recursion}
For \(1\leq n\leq N\),
\begin{equation}\label{lower:eq:recursion}
f_n=\mathcal K_nf_{n-1}.
\end{equation}
Consequently,
\begin{equation}\label{lower:eq:transform-recursion}
f_N(0)
=
\left(
\mathcal K_N\mathcal K_{N-1}\cdots\mathcal K_1\mathbf 1
\right)(0)
=
\E[e^{-sA_\gamma/\gamma^2}].
\end{equation}
\end{proposition}

\begin{proof}
We first derive the Brownian-clock representation of a single block.

\medskip

\noindent
\emph{Step 1: From the cone to the strip.}
Consider the conformal map
\[
z\longmapsto w=\log(z\sin\theta)
\]
from the cone \(\{z:0<\arg z<\theta\}\) onto the strip
\(\mathbb R\times(0,\theta)\). An unbiased cone excursion that starts at
radius \(1/\sin\theta\) on the lower side of the cone and ends at radius
\(\ell/\sin\theta\) on the upper side is mapped, in conformal time, to a
strip-crossing excursion.

Write the strip process as \((B_u,Y_u)\), where \(B\) is a standard
Brownian motion and \(Y\) is the vertical strip-crossing excursion. The
two coordinates are independent. Let \(U\) be the crossing time of the
vertical coordinate. Since the terminal point in the strip is
\(B_U+i\theta=\log\ell+i\theta\), we have
\begin{equation}\label{lower:eq:ell-strip}
\ell=e^{B_U}.
\end{equation}

The inverse conformal map is \(g(w)=e^w/\sin\theta\), and hence
\[
|g'(B_u+iY_u)|^2
=
\frac{e^{2B_u}}{\sin^2\theta}.
\]
The conformal change of clock therefore gives
\begin{equation}\label{lower:eq:cone-time}
T
=
\frac1{\sin^2\theta}
\int_0^Ue^{2B_u}\dd u.
\end{equation}

\medskip

\noindent
\emph{Step 2: The strip-crossing time.}
We claim that
\begin{equation}\label{lower:eq:strip-clock-law}
U\overset{d}=\frac{2\theta^2}{\pi^2}V_0.
\end{equation}
To prove this, let \(Y\) be a one-dimensional Brownian motion started at
\(y\in(0,\theta)\), and define
\[
\tau_0=\inf\{u\geq0:Y_u=0\},
\qquad
\tau_\theta=\inf\{u\geq0:Y_u=\theta\}.
\]
The strip-crossing excursion is obtained as the entrance limit
\(y\downarrow0\) of \(Y\) conditioned on \(\tau_\theta<\tau_0\).
Therefore,
\[
\E[e^{-\lambda U}]
=
\lim_{y\downarrow0}
\E_y\left[
e^{-\lambda\tau_\theta}
\,\middle|\,
\tau_\theta<\tau_0
\right].
\]

Let
\(q_\lambda(y)=
\E_y[e^{-\lambda\tau_\theta};\tau_\theta<\tau_0]\).
The one-dimensional Feynman--Kac formula gives
\[
\frac12q_\lambda''=\lambda q_\lambda
\quad\text{on }(0,\theta),
\qquad
q_\lambda(0)=0,
\qquad
q_\lambda(\theta)=1.
\]
Thus
\[
q_\lambda(y)
=
\frac{\sinh(y\sqrt{2\lambda})}
{\sinh(\theta\sqrt{2\lambda})}.
\]
Since \(\mathbb{P}_y(\tau_\theta<\tau_0)=y/\theta\), it follows that
\begin{equation}\label{lower:eq:strip-clock-laplace}
\E[e^{-\lambda U}]
=
\frac{\theta\sqrt{2\lambda}}
{\sinh(\theta\sqrt{2\lambda})}.
\end{equation}

On the other hand, independence of the \(E_j\)'s and Euler's product
formula give
\[
\E[e^{-tV_0}]
=
\prod_{j=1}^{\infty}\frac{j^2}{j^2+t}
=
\frac{\pi\sqrt t}{\sinh(\pi\sqrt t)}.
\]
Taking \(t=2\lambda\theta^2/\pi^2\), we obtain
\[
\E\left[
e^{-\lambda(2\theta^2V_0/\pi^2)}
\right]
=
\frac{\theta\sqrt{2\lambda}}
{\sinh(\theta\sqrt{2\lambda})}.
\]
Comparison with \eqref{lower:eq:strip-clock-laplace} proves
\eqref{lower:eq:strip-clock-law}.

\medskip

\noindent
\emph{Step 3: Incorporating the boundary bias.}
The block appearing at the \(n\)-th step of the recursion has index
\(k=N+1-n\). Its boundary weight is
\[
\ell_k^{-(N-k)}=\ell_k^{-(n-1)}.
\]
By \eqref{lower:eq:ell-strip},
\[
\ell_k^{-(n-1)}
=
e^{-(n-1)B_U}.
\]
Factor this weight as
\begin{equation}\label{lower:eq:girsanov-factor}
e^{-(n-1)B_U}
=
e^{-(n-1)B_U-\frac12(n-1)^2U}
e^{\frac12(n-1)^2U}.
\end{equation}
The first factor on the right is the exponential martingale that changes
the horizontal Brownian motion into a Brownian motion with drift
\(-(n-1)\). By \eqref{lower:eq:strip-clock-law}, the second factor is
\[
e^{\frac12(n-1)^2U}
=
\exp\left\{
\left(\frac{(n-1)\theta}{\pi}\right)^2V_0
\right\}.
\]
It therefore changes \(V_0\) into \(V_n\), as defined in
\eqref{lower:eq:tilted-clock-law}. Its normalizing constant is
\[
H\left(\left(\frac{(n-1)\theta}{\pi}\right)^2\right).
\]

After this change of measure, the horizontal coordinate in conformal
time is \(B_u-(n-1)u\). Using Brownian scaling and the substitution
\[
u=\frac{2\theta^2}{\pi^2}v,
\]
we obtain
\[
B_{2\theta^2v/\pi^2}
-(n-1)\frac{2\theta^2}{\pi^2}v
\overset{d}=
\frac{\sqrt2\,\theta}{\pi}B_v
-\frac{2(n-1)\theta^2}{\pi^2}v.
\]
Thus, if the logarithm of the accumulated boundary scale at the
beginning of the block is \(x\), the logarithmic scale during that block
is precisely the process \(X^{(n)}\) in
\eqref{lower:eq:block-log-process}. In particular,
\begin{equation}\label{lower:eq:block-endpoint}
x+\log\ell_k=X_{V_n}^{(n)}.
\end{equation}

Similarly, \eqref{lower:eq:cone-time} becomes
\begin{equation}\label{lower:eq:intrinsic-block-time}
T_k
=
\frac{2\theta^2}{\pi^2\sin^2\theta}
\int_0^{V_n}
\exp\left\{
2\left(
\frac{\sqrt2\,\theta}{\pi}B_v
-\frac{2(n-1)\theta^2}{\pi^2}v
\right)
\right\}
\dd v.
\end{equation}
Multiplying by the square \(e^{2x}\) of the accumulated boundary scale
gives
\begin{equation}\label{lower:eq:scaled-block-area}
e^{2x}T_k
=
\frac{2\theta^2}{\pi^2\sin^2\theta}
\int_0^{V_n}e^{2X_v^{(n)}}\dd v.
\end{equation}

\medskip

\noindent
\emph{Step 4: The Laplace coefficient and the recursion.}
Combining \eqref{lower:eq:scaled-block-area} with the coefficient in
\eqref{lower:eq:preliminary-recursion}, we obtain
\[
\frac{s\sin\theta}{2\gamma^2}e^{2x}T_k
=
\frac{s\theta^2}{\pi^2\gamma^2\sin\theta}
\int_0^{V_n}e^{2X_v^{(n)}}\dd v.
\]
Since \(\theta=\pi\gamma^2/4\),
\[
\frac{\theta^2}{\pi^2\gamma^2}
=
\frac{\gamma^2}{16}.
\]
Therefore,
\begin{equation}\label{lower:eq:block-laplace-factor}
\frac{s\sin\theta}{2\gamma^2}e^{2x}T_k
=
\frac{s\gamma^2}{16\sin\theta}
\int_0^{V_n}e^{2X_v^{(n)}}\dd v.
\end{equation}

Using \eqref{lower:eq:block-endpoint} and
\eqref{lower:eq:block-laplace-factor}, the conditional recursion
\eqref{lower:eq:preliminary-recursion} becomes
\[
f_n(x)
=
\E\left[
\exp\left\{
-\frac{s\gamma^2}{16\sin\theta}
\int_0^{V_n}e^{2X_v^{(n)}}\dd v
\right\}
f_{n-1}\left(X_{V_n}^{(n)}\right)
\right].
\]
By the definition of \(\mathcal K_n\), this is exactly
\(f_n=\mathcal K_nf_{n-1}\). Iterating the recursion proves
\eqref{lower:eq:transform-recursion}.
\end{proof}

The recursion index \(n\) increases from \(1\) to \(N\), while the
physical block index \(k=N+1-n\) decreases from \(N\) to \(1\). Thus the
innermost operator \(\mathcal K_1\) corresponds to the final block, whose
boundary bias is zero, while the outermost operator \(\mathcal K_N\)
corresponds to the first physical block.

\section{Hamilton--Jacobi equation for moments}\label{upper:sec:moment-recursion}

The exact recurrence is now available. We normalize its weights, identify
the limiting equation, and construct a comparison function. Convergence
is proved separately in Section~\ref{sec:convergence}.

\subsection{A normalization that controls all exponents}

Fix $\max\{W/2,q/2\}<B<1$. For all sufficiently small $\gamma$,
every weight $w_\theta(b,d)$ occurring in the recurrence satisfies
$|b+2h\theta/\pi|\le B$ for $0\le h\le d$, and also
$|b+(2d-1)\theta/\pi|\le B$.
Indeed, here $b=(2(j-d)-(m-1))\theta/\pi$, with
$1\le m\le N$ and $0\le d\le j\le n$, so all these quantities
lie between $-N\theta/\pi=-W/2$ and $2n\theta/\pi\to q/2$.

For $d\ge1$, the divided-difference identity proved in
Lemma~\ref{lem:mixed-block-moments}, applied to the inputs
$(b+2h\theta/\pi)^2$, $h=0,\ldots,d$, gives
\[
 w_\theta(b,d)=2\sum_{k\ge1}
 \frac{(-1)^{k+1}k^2}{\prod_{h=0}^d[k^2-(b+2h\theta/\pi)^2]}.
\]
The terms in this series alternate in sign. Their absolute values are
\[
 \frac{2k^2}{\prod_{h=0}^d[k^2-(b+2h\theta/\pi)^2]}
 =2k^{-2d}\prod_{h=0}^d
 \left(1-\frac{(b+2h\theta/\pi)^2}{k^2}\right)^{-1}.
\]
Every factor on the right is positive and nonincreasing as $k$
increases, and $k^{-2d}$ tends to zero. Thus these absolute values
decrease to zero. It follows that keeping only the $k=1$ term gives
an upper bound for the series, while keeping the $k=1$ and $k=2$
terms gives a lower bound.
Moreover, the absolute value of the $k=2$ term divided by that of
the $k=1$ term is
\[
 4\prod_{h=0}^d
 \frac{1-(b+2h\theta/\pi)^2}{4-(b+2h\theta/\pi)^2}
 \le 4\left(\frac14\right)^{d+1}
 =4^{-d}\le\frac14.
\]
Here each fraction is at most $1/4$, since
$0\le(b+2h\theta/\pi)^2<1$.
The lower bound is therefore at least $3/4$ of the positive $k=1$
term. Since that term equals
$2\prod_{h=0}^d[1-(b+2h\theta/\pi)^2]^{-1}$, we obtain
\begin{equation}\label{upper:eq:weight-product-bounds}
 \frac32\prod_{h=0}^d[1-(b+2h\theta/\pi)^2]^{-1}
 \le w_\theta(b,d)\le
 2\prod_{h=0}^d[1-(b+2h\theta/\pi)^2]^{-1}.
\end{equation}
This estimate is valid even when $d$ is of order $\gamma^{-2}$.

Define
\[
 g(b)=-\log(1-b^2),\qquad V(b)=\frac12\int_0^b g(v)\,dv.
\]
One block changes $b$ to $b+(2d-1)\theta/\pi$. Introduce
\[
 \widetilde w_\theta(b,d)
 =w_\theta(b,d)\exp\{-\frac{4}{\gamma^2}[V(b+(2d-1)\theta/\pi)-V(b)]\}.
\]
Here admissibility means that
$|b+2h\theta/\pi|\le B$ for $0\le h\le d$ and
$|b+(2d-1)\theta/\pi|\le B$.
Define
\[
 a_d(b)=(1-b^2)^{d-1/2}\frac{H^{(d)}(b^2)}{d!},
 \qquad A(b,z)=\sum_{d\ge0}a_d(b)z^d,
 \quad |b|<1,\quad 0<z<1.
\]

\begin{lemma}[Uniform weight bounds]\label{upper:lem:uniform-local-weights}
There are constants $0<c_B<C_B<\infty$ such that
\begin{equation}\label{upper:eq:uniform-transformed-weights}
 c_B\le\widetilde w_\theta(b,d)\le C_B
\end{equation}
for every admissible $b,d,\theta$.
\end{lemma}

\begin{proof}
By \eqref{upper:eq:weight-product-bounds}, the logarithm of
$w_\theta(b,d)$ differs by a bounded amount from
$\sum_{h=0}^d g(b+2h\theta/\pi)$. For $d\geq1$, the elementary
Riemann-sum estimate is
\[
 \left|\sum_{h=0}^d g(b+2h\theta/\pi)
 -\frac{\pi}{2\theta}\int_b^{b+2d\theta/\pi}g(y)\,dy\right|
 \leq \|g\|_\infty+\frac{d\theta}{\pi}\|g'\|_\infty,
\]
where the norms are on $[-B,B]$. The right-hand side is bounded
uniformly because admissibility bounds $d\theta/\pi$.
Changing the upper integration endpoint to $b+(2d-1)\theta/\pi$
costs at most $\|g\|_\infty/2$. Hence the error between the sum and
\[
 \frac{4}{\gamma^2}\{V(b+(2d-1)\theta/\pi)-V(b)\}
\]
is bounded uniformly in $b,d,\theta$. This proves
\eqref{upper:eq:uniform-transformed-weights} for $d\geq1$.
For $d=0$, use $w_\theta(b,0)=H(b^2)$ and the bounded difference
$\frac{\pi}{\theta}\{V(b-\theta/\pi)-V(b)\}$.

\end{proof}

\begin{lemma}[Limits of the weights]\label{upper:lem:weight-limits}
For fixed $d$, locally uniformly in $|b|<1$,
\begin{equation}\label{upper:eq:local-weight-limit}
 \widetilde w_\theta(b,d)\longrightarrow a_d(b).
\end{equation}
On $|b|\le B$, the functions $a_d$ are bounded above and away from zero,
and the derivatives $(\log a_d)'$ are bounded, uniformly in $d$.
\end{lemma}

\begin{proof}
For fixed $d$, the inputs to the divided difference tend to $b^2$, while
the exponential normalization converges to $(1-b^2)^{d-1/2}$.  This
proves \eqref{upper:eq:local-weight-limit}.  The bounds for $a_d$ follow
by taking the fixed-$d$ limit in the uniform bounds above, with constants
independent of $d$.

For $d\geq1$, the resolvent representation gives
\[
 a_d(b)=2(1-b^2)^{-3/2}\sum_{k\geq1}(-1)^{k+1}k^2
 \left(\frac{1-b^2}{k^2-b^2}\right)^{d+1}.
\]
The $k=1$ factor inside the sum is identically one.  For $k\geq2$ and
$|b|\leq B$, both the terms and their $b$-derivatives are dominated by
a summable sequence, uniformly in $d\geq1$. Indeed, if
$R_k(b)=(1-b^2)/(k^2-b^2)$, then $R_k\leq k^{-2}$ and
$|R_k'|\leq C_B R_k$. Thus both $k^2R_k^{d+1}$ and
$k^2(d+1)R_k^d|R_k'|$ are bounded by $C_B/k^2$ for
$k\geq2$, $d\geq1$. Thus $a_d'$ is uniformly bounded.  The
uniform positive lower bound for $a_d$ then gives a uniform bound for
$(\log a_d)'$.  The case $d=0$ is a single smooth positive function.
This proves the asserted uniform derivative bound.
\end{proof}

Fix a common modulus of continuity $\omega_B(\delta)\downarrow0$ such that
\begin{equation}\label{upper:eq:uniform-log-continuity}
 \sup_{d\ge0}\sup_{\substack{|b|,|\widetilde b|\le B\\
 |b-\widetilde b|\le\delta}}
 |\log a_d(b)-\log a_d(\widetilde b)|\le\omega_B(\delta).
\end{equation}
Its existence follows from Lemma~\ref{upper:lem:weight-limits}.

\begin{lemma}[Sum of the limiting weights]\label{upper:lem:weight-sum}
For $|b|<1$ and $0<z<1$,
\begin{equation}\label{upper:eq:local-generating-function}
 A(b,z)=(1-b^2)^{-1/2}H(b^2+(1-b^2)z).
\end{equation}
\end{lemma}

\begin{proof}
Taylor's formula for $H$ at $b^2$ gives
\begin{align*}
 \sum_{d\geq0}a_d(b)z^d
 &=(1-b^2)^{-1/2}
   \sum_{d\geq0}\frac{H^{(d)}(b^2)}{d!}
       \{(1-b^2)z\}^d,
\end{align*}
which is \eqref{upper:eq:local-generating-function}.
\end{proof}

\subsection{The limiting equation}

Define the normalized array
\[
 \widetilde Q_{m,j}
 =
 \exp\left\{-\frac4{\gamma^2}
 \left[V\left(\frac{(2j-m)\theta}{\pi}\right)-V(0)\right]\right\}
 Q_{m,j}.
\]
The normalization telescopes along the recurrence, giving the exact identity
\begin{equation}\label{upper:eq:normalized-recursion}
 \widetilde Q_{m,j}
 =
 \sum_{d=0}^j
 \widetilde w_\theta\left(
 \frac{(2(j-d)-(m-1))\theta}{\pi},d\right)
 \widetilde Q_{m-1,j-d}.
\end{equation}

Set $t=m\theta/\pi$, $u=j\theta/\pi$, and formally try
$\widetilde Q_{m,j}\approx e^{4G(t,u)/\gamma^2}$.
For fixed $d$, the ratio of the previous exponential to the current
one tends to $e^{-G_t-dG_u}$, and the weight tends to $a_d(2u-t)$.
Summing suggests
\begin{equation}\label{upper:eq:normalized-hj}
 G_t=\log A(2u-t,e^{-G_u}),\qquad G_u>0.
\end{equation}
The condition $G_u>0$ ensures convergence of the generating function.
Equivalently, $F(t,u)=G(t,u)+V(2u-t)$ satisfies
\begin{equation}\label{upper:eq:moment-hj}
 F_t=\log H\bigl((2u-t)^2+e^{-F_u}\bigr),\qquad
 (2u-t)^2+e^{-F_u}<1.
\end{equation}
This calculation only identifies a candidate equation; the discrete
comparison in Section~\ref{sec:convergence} will justify convergence.

\subsection{Construction of the comparison function}

We construct the candidate explicitly. In this subsection
$h(v)=H'(v)/H(v)$, and the variable $s$ parametrizes a characteristic.

For $-1<C<1$, set
\begin{align*}
 y_C(s)&=\frac1\pi\arccos(C\cos(\pi s)),\\
 b_C(s)&=-\frac{y_C(s)C\sin(\pi s)}
                   {\sqrt{1-C^2\cos^2(\pi s)}},&
 u_C(s)&=\frac{s+b_C(s)}2,\\
 z_C(s)&=y_C(s)^2-b_C(s)^2
       =\frac{y_C(s)^2(1-C^2)}{1-C^2\cos^2(\pi s)}.
\end{align*}
For $0<t<1$ and $0<u<1/2$, let $C=C(t,u)$ solve $u_C(t)=u$;
existence and uniqueness are proved below. Using this parameter, define
\begin{equation}\label{upper:eq:characteristic-integral}
 F(t,u)=\int_0^t
 \left[\log H(y_C(s)^2)-u_C'(s)\log z_C(s)\right]\,ds,
 \qquad G(t,u)=F(t,u)-V(2u-t).
\end{equation}
\begin{lemma}[Moment comparison function]\label{upper:lem:hj-candidate}
The parameter $C(t,u)$ is unique. The functions $F,G$ are real analytic
on $(0,1)\times(0,1/2)$ and satisfy
\eqref{upper:eq:moment-hj} and \eqref{upper:eq:normalized-hj}.
\end{lemma}

\begin{proof}
\emph{Characteristics and endpoints.}
The characteristic equations for \eqref{upper:eq:moment-hj}, with
$b=2u-s$ and $z=e^{-F_u}$, are
\[
 u'=zh(b^2+z),\qquad z'=-4bzh(b^2+z).
\]
They give $b'=2zh-1$ and, for $y=\sqrt{b^2+z}$,
$y'=-b/y$.
Using $1-2y^2h(y^2)=\pi y\cot(\pi y)$ yields
$(\cos(\pi y))''+\pi^2\cos(\pi y)=0$.
The initial condition $u(0)=0$ gives $b(0)=0$, hence $y'(0)=0$.
Solving this scalar equation gives the formulas defining $y_C,b_C,u_C,z_C$.
They also directly verify
\[
 u_C(0)=0,\qquad u_C'=z_Ch(y_C^2)>0,\qquad
 0<z_C<1-b_C^2.
\]

For fixed $0<t<1$, write $x=C\cos(\pi t)$. Differentiation gives
\[
 \partial_C b_C(t)=
 -\frac{\sin(\pi t)}{\pi(1-x^2)^{3/2}}
 \left[\arccos x-x\sqrt{1-x^2}\right]<0.
\]
The bracket is positive: its derivative is $-2\sqrt{1-x^2}$
and its limit at $x=1$ is zero.
Moreover, $b_C(t)\to-t$ as $C\uparrow1$ and
$b_C(t)\to1-t$ as $C\downarrow-1$.
Thus $C\mapsto u_C(t)$ is an analytic bijection onto $(0,1/2)$
with nonzero derivative. Its inverse is analytic.

\emph{Direct verification of the derivatives.}
Put $p_C=-\log z_C$, so that $p_C'=4b_Ch(y_C^2)$.
Differentiating the integrand in
\eqref{upper:eq:characteristic-integral} with respect to $C$ gives
\[
 \partial_C\{\log H(b_C^2+z_C)+p_Cu_C'\}
 =
 \frac{d}{ds}\{p_C\,\partial_Cu_C\}.
\]
Indeed, the terms containing $\partial_Cp_C$ cancel because
$u_C'=z_Ch(y_C^2)$, and the remaining terms use
$p_C'=4b_Ch(y_C^2)$.
Since $\partial_Cu_C(0)=0$, integration gives
$F_u(t,u_C(t))=p_C(t)$.
Differentiating along the characteristic then gives
$F_t=\log H(y_C^2)$.
These identities prove the unnormalized HJ equation. Subtracting
$V(2u-t)$ gives the normalized equation and the formula for $G_u$.
Analyticity follows from the integral and the analytic inverse $C(t,u)$.

\begin{equation}\label{upper:eq:characteristic-derivatives}
 F_u(t,u)=-\log z_C(t),\qquad F_t(t,u)=\log H(y_C(t)^2),
 \qquad
 G_u(t,u)=\log\frac{1-(2u-t)^2}{z_C(t)}>0.
\end{equation}
\end{proof}

\begin{lemma}[Boundary values of the moment comparison function]
\label{upper:lem:hj-boundary-values}
For $T<1$ and $R<1/2$, $G$ extends continuously to
$[0,T]\times[0,R]$, with
\begin{equation}\label{upper:eq:hj-boundaries}
 G(0,u)=0,\qquad G(t,0)=\int_0^t\log a_0(-x)\,dx.
\end{equation}
\end{lemma}

\begin{proof}
We estimate only the explicit characteristic integral; no comparison
with other paths is required.
Along a characteristic write $b=b_C(s)$, $v=u_C'(s)>0$, and
$\zeta=z_C(s)/(1-b^2)\in(0,1)$. The identity for $A$ gives
\[
 v=\frac{\zeta A_z(b,\zeta)}{A(b,\zeta)},\qquad
 G(t,u)=\int_0^t
 \{\log A(b,\zeta)-v\log\zeta\}\,ds.
\]
The latter follows by subtracting
$\frac{d}{ds}V(b_C(s))=g(b_C(s))(v-1/2)$ from the integrand of $F$.

The function $\log A(b,e^p)$ is convex in $p<0$.
The displayed identity for $v$ therefore says that $\zeta$ minimizes
$\log A(b,\xi)-v\log\xi$ over $0<\xi<1$.
For endpoints in $[0,T]\times[0,R]$, monotonicity of $u_C$ implies
$-T\leq b_C(s)\leq2R$, a compact subinterval of $(-1,1)$.
On this state interval, the coefficient bounds imply
$A(b,\xi)\leq C_B/(1-\xi)$ and
$a_d(b)\leq a_0(b)e^{L_Bd}$ for a constant $L_B$.
Evaluating at $\xi=v/(1+v)$ and
$\xi=e^{-L_B}v/(1+v)$, respectively, gives
\begin{align*}
 0\leq\log A(b,\zeta)-v\log\zeta
 &\leq\log C_B+(1+v)\log(1+v)-v\log v,\\
 0\leq\log A(b,\zeta)-v\log\zeta-\log a_0(b)
 &\leq L_Bv+(1+v)\log(1+v)-v\log v.
\end{align*}
The lower bounds use $A(b,\zeta)\geq a_0(b)\geq1$.

The characteristic has $\int_0^t v\,ds=u$ and
$0\leq u_C(s)\leq u$. Concavity of
$(1+v)\log(1+v)-v\log v$ and Jensen's inequality imply
\[
 0\leq G(t,u)\leq
 t\log C_B+(t+u)\log(t+u)-t\log t-u\log u.
\]
Also $|b_C(s)+s|\leq2u$, so smoothness of $\log a_0$ gives
\[
 \left|G(t,u)-\int_0^t\log a_0(-s)\,ds\right|
 \leq L_Bu+(t+u)\log(t+u)-t\log t-u\log u.
\]
Here and below $0\log0=0$. The first bound tends uniformly to zero
as $t\downarrow0$ with $u$ bounded; the second tends uniformly
to zero as $u\downarrow0$ with $t$ bounded.
They agree at the origin and prove the asserted continuous extension.
\end{proof}

\section{Hamilton--Jacobi equation for the Laplace transform}\label{lower:sec:hj}

\subsection{Formal derivation of the Hamilton--Jacobi equation}

We formally derive the Hamilton--Jacobi equation suggested by the exact
recursion. Set \(t_n=(n-1)\theta/\pi\). Since
\(f_n=\mathcal K_nf_{n-1}\), we make the ansatz
\[
f_{n-1}(x)
\approx
\exp\left\{\frac{4}{\gamma^2}\mathcal G(t_n,x)\right\}.
\]
Notice that
\(t_{n+1}-t_n=\theta/\pi=\gamma^2/4\).

Recall that
\[
X_v^{(n)}
=
x+\frac{\sqrt2\,\theta}{\pi}B_v
-\frac{2(n-1)\theta^2}{\pi^2}v
=
x+\frac{\gamma^2}{2\sqrt2}B_v
-\frac{\gamma^2t_n}{2}v.
\]
Applying \(\mathcal K_n\) to the exponential ansatz gives
\begin{align}
&\frac{
\mathcal K_n
\exp\left\{\frac{4}{\gamma^2}\mathcal G(t_n,\cdot)\right\}(x)
}{
\exp\left\{\frac{4}{\gamma^2}\mathcal G(t_n,x)\right\}
}
\notag\\
&\quad=
\E\left[
\exp\left\{
-\frac{s\gamma^2}{16\sin\theta}
\int_0^{V_n}e^{2X_v^{(n)}}\dd v
+
\frac{4}{\gamma^2}
\left(
\mathcal G(t_n,X_{V_n}^{(n)})
-\mathcal G(t_n,x)
\right)
\right\}
\right].
\label{lower:eq:formal-operator-ratio}
\end{align}

As \(\gamma\to0\), we have
\(s\gamma^2/(16\sin\theta)\to s/(4\pi)\). Moreover,
\(X_v^{(n)}-x=O(\gamma^2)\) on bounded clock intervals, so formally
\[
\frac{s\gamma^2}{16\sin\theta}
\int_0^{V_n}e^{2X_v^{(n)}}\dd v
\approx
\frac{s}{4\pi}e^{2x}V_n.
\]
A first-order Taylor expansion also gives
\[
\frac{4}{\gamma^2}
\left(
\mathcal G(t_n,X_{V_n}^{(n)})
-\mathcal G(t_n,x)
\right)
\approx
\sqrt2\,\mathcal G_x(t_n,x)B_{V_n}
-2t_n\mathcal G_x(t_n,x)V_n.
\]

Write \(p=\mathcal G_x(t_n,x)\). Conditional on \(V_n=v\),
\(\E[e^{\sqrt2pB_v}]=e^{p^2v}\). Therefore,
\[
\frac{
\mathcal K_n
\exp\left\{\frac{4}{\gamma^2}\mathcal G(t_n,\cdot)\right\}(x)
}{
\exp\left\{\frac{4}{\gamma^2}\mathcal G(t_n,x)\right\}
}
\approx
\E\left[
e^{\left(p^2-2t_np-\frac{s}{4\pi}e^{2x}\right)V_n}
\right].
\]
Since
\[
p^2-2t_np-\frac{s}{4\pi}e^{2x}
=
(p-t_n)^2-t_n^2-\frac{s}{4\pi}e^{2x},
\]
the clock moment-generating formula yields
\begin{equation}\label{lower:eq:formal-one-step-limit}
\frac{
\mathcal K_n
\exp\left\{\frac{4}{\gamma^2}\mathcal G(t_n,\cdot)\right\}(x)
}{
\exp\left\{\frac{4}{\gamma^2}\mathcal G(t_n,x)\right\}
}
\approx
\frac{
H\left(
(\mathcal G_x(t_n,x)-t_n)^2
-\frac{s}{4\pi}e^{2x}
\right)
}{
H(t_n^2)
}.
\end{equation}

On the other hand, the recursion predicts that the same ratio is
approximately
\[
\exp\left\{
\frac{4}{\gamma^2}
\left(
\mathcal G(t_{n+1},x)-\mathcal G(t_n,x)
\right)
\right\}
\approx
e^{\mathcal G_t(t_n,x)},
\]
because \(t_{n+1}-t_n=\gamma^2/4\). Matching this expression with
\eqref{lower:eq:formal-one-step-limit} gives
\begin{equation}\label{lower:eq:comparisonpde}
\boxed{
\mathcal G_t(t,x)
=
\log H\left(
(\mathcal G_x(t,x)-t)^2
-\frac{s}{4\pi}e^{2x}
\right)
-\log H(t^2).
}
\end{equation}
The initial condition is \(\mathcal G(0,x)=0\), since \(f_0(x)=1\).

Finally, the clock moment-generating function is finite only when its
argument is smaller than \(1\). We therefore seek a solution satisfying
\begin{equation}\label{lower:eq:strictdomain}
(\mathcal G_x(t,x)-t)^2
-\frac{s}{4\pi}e^{2x}
<1.
\end{equation}
This derivation is formal because \(V_n\) is unbounded. The comparison
argument below will use \eqref{lower:eq:strictdomain} to justify the
one-step approximation and control the large-clock contribution.

\subsection{Construction of the comparison function}

Fix $s>0$. The function $h=H'/H$ and the inverse $\chi^{-1}$ were
defined in Section~\ref{sec:common-blocks}. Define
$\mathcal H(x,p)=-x+\log H(p^2-\frac{s}{4\pi}e^{2x})$.
For $C>1$ and times with $C\cos(\pi t)>-1$, set
\begin{align}
 v_C(t)&=\chi^{-1}(C\cos(\pi t)),\label{lower:eq:charv}\\
 p_C(t)&=-\frac{C\sin(\pi t)H(v_C(t))}{\pi},\label{lower:eq:charp}\\
 z_C(t)&=v_C(t)-p_C(t)^2
       =\frac{1-C^2}{\pi^2}H(v_C(t))^2,\label{lower:eq:charz}\\
 x_C(t)&=\frac12\log\frac{-4\pi z_C(t)}s.\label{lower:eq:charx}
\end{align}
For each $(t,x)$ with $0\le t<1$, choose $C$ so that $x_C(t)=x$;
the following lemma proves that this choice is unique.
Define $\mathcal U$ and then $\mathcal G$ by
\[
 \mathcal U(t,x_C(t))=\int_0^t
 \{\mathcal H(x_C(a),p_C(a))+p_C(a)x_C'(a)\}\dd a.
\]
\[
 \mathcal G(t,x)=\mathcal U(t,x)+tx-\int_0^t\log H(a^2)\,\dd a.
\]

\begin{lemma}[Laplace comparison function]\label{lower:lem:hj-candidate}
The parameter $C$ is unique and $\mathcal G$ is smooth on
$[0,1)\times\mathbb R$. It solves \eqref{lower:eq:comparisonpde},
has initial value zero, and satisfies \eqref{lower:eq:strictdomain}.
\end{lemma}

\begin{proof}
The substitution defining $\mathcal G$ turns the desired equation into
$\mathcal U_t=\mathcal H(x,\mathcal U_x)$.
Its characteristic equations, with initial momentum $p(0)=0$, are
\[
 x'=-2p h(v),\qquad
 p'=-1-\frac{s}{2\pi}e^{2x}h(v),\qquad
 v=p^2-\frac{s}{4\pi}e^{2x}.
\]
They imply $v'=-2p$ and
$(\chi(v(t)))''+\pi^2\chi(v(t))=0$.
The initial derivative of $\chi(v(t))$ is zero, so
$\chi(v_C(t))=C\cos(\pi t)$. Differentiating this identity gives
\eqref{lower:eq:charp}; the remaining formulas follow from
$v=p^2-\frac{s}{4\pi}e^{2x}$.
At $t=0$, $v_C(0)=-(\operatorname{arcosh}C/\pi)^2$ and $p_C(0)=0$.

For each $0\le t<1$, the map $C\mapsto-z_C(t)$ is strictly increasing
from zero to infinity on its permitted interval. For $t>1/2$ this follows
because $v_C(t)$ and $H(v_C(t))$ both increase with $C$; for $t=1/2$,
$-z_C(t)=(C^2-1)/4$. For $0<t<1/2$, write $\vartheta=\pi t$.
On the positive-$v$ branch, with $a=\pi\sqrt v\in(0,\vartheta)$,
\[
 -z=\frac{a^2}{\pi^2}
       \{\tan^2\vartheta\cot^2a-1\}.
\]
Its derivative in $a$ is negative because
$\cot a-a\csc^2a<0$. On the negative-$v$ branch, with
$a=\pi\sqrt{-v}>0$,
\[
 -z=\frac{a^2}{\pi^2}
       \{1+\tan^2\vartheta\coth^2a\},
\]
whose derivative is positive because $\coth a-a\operatorname{csch}^2a>0$.
The branches join analytically with nonzero derivative at $v=0$.
The case $t=0$ follows directly from $\operatorname{arcosh}C$.

Consequently, the curves cover each $(t,x)$ exactly once and the
parameter depends smoothly on the endpoint.
In the defining integral for $\mathcal U$, the initial endpoint is
free and its momentum is zero. The first-variation formula therefore
gives $\mathcal U_x=p_C(t)$ and
$\mathcal U_t=\mathcal H(x,\mathcal U_x)$.
The definition of $\mathcal G$ now gives the HJ equation and
$\mathcal G(0,x)=0$.
Its argument of $H$ is $v_C(t)<1$, proving the strict-domain condition.
\end{proof}

In particular, on each $[0,W/2]\times[-M,M]$, all required derivatives
are bounded and the argument of $H$ is at most $1-\delta_M$ for some
$\delta_M>0$. These are the compact bounds used in the convergence proof.

\section{Convergence of the moment recursion}\label{sec:convergence}

Lemmas~\ref{upper:lem:hj-candidate} and~\ref{upper:lem:hj-boundary-values}
supply the comparison function $G$ and its boundary values.
We now compare its exponential with the exact normalized recurrence.

Two estimates suffice. Away from the axes, the exponential nearly
satisfies one recurrence step. Near the axes, the discrete array
approaches the boundary values of $G$.
Positivity propagates the resulting bounds by induction.

\subsection{One-step comparison}

Fix $W/2<T<1$, $q/4<R<1/2$, and then a sufficiently small $\eta>0$,
so that all shifted grid points considered below lie in
$[\eta,T]\times[\eta,R]$.
For $t=m\theta/\pi$, $u=j\theta/\pi$, define
\begin{align}
 R_{m,j}^{\eta}
 =\sum_{d=0}^j&
 \widetilde w_\theta\left(
 2u-t+(1-2d)\frac{\theta}{\pi},d\right)
 \notag\\[-2pt]
 &\times\exp\left\{\frac4{\gamma^2}
 \left[
 G\left(t-\frac{\theta}{\pi}+\eta,
        u-\frac{d\theta}{\pi}+\eta\right)
 -G(t+\eta,u+\eta)
 \right]\right\}.
 \label{upper:eq:hj-test-ratio}
\end{align}
This is the ratio between applying the exact recursion to the previous
row of the test array and its value at the current grid point.

Choose $B<1$ containing all actual and shifted values of $2u-t$.

\begin{lemma}[Uniform one-step bounds]\label{upper:lem:hj-consistency}
For fixed $\eta>0$, there are numbers
$\varepsilon_{\gamma,\eta}\geq0$ tending to zero as
$\gamma\downarrow0$ such that, uniformly for $2\leq m\leq N$ and
$0\leq j\leq n$,
\[
 R_{m,j}^{\eta}\leq
 \exp\{\omega_B(\eta)+\varepsilon_{\gamma,\eta}\}.
\]
If also $j\theta/\pi\geq\eta$, then
\[
 R_{m,j}^{\eta}\geq
 \exp\{-\omega_B(\eta)-\varepsilon_{\gamma,\eta}\}.
\]
\end{lemma}

\begin{proof}
By Lemma~\ref{upper:lem:hj-candidate}, on the shifted rectangle
$G_u\geq p_\eta>0$ and $|G_t|\leq M_\eta$.
Integrating these derivatives along a horizontal and then a vertical
segment bounds the exponential in
\eqref{upper:eq:hj-test-ratio} by $e^{M_\eta-p_\eta d}$.
The exact normalized weights are uniformly bounded, so the terms
with $d>K$ have total at most
$C_Be^{M_\eta}\sum_{d>K}e^{-p_\eta d}$.
This bound holds even for exponents of order $\gamma^{-2}$.

Fix $K$ independently of $\gamma$. Bounded second derivatives give
\[
 \frac4{\gamma^2}
 [G(t-\theta/\pi+\eta,u-d\theta/\pi+\eta)-G(t+\eta,u+\eta)]
 =-G_t-dG_u+O_\eta(\gamma^2(1+K)^2)
\]
uniformly for $0\le d\le\min\{j,K\}$. No Taylor expansion is used for
$d>K$. Together with Lemma~\ref{upper:lem:weight-limits}, this gives
\[
 \widetilde w_\theta\left(
 2u-t+(1-2d)\frac{\theta}{\pi},d\right)
 e^{\frac4{\gamma^2}
 [G(t-\theta/\pi+\eta,u-d\theta/\pi+\eta)-G(t+\eta,u+\eta)]}
 =
 a_d(2u-t)e^{-G_t-dG_u}+o(1),
\]
where the derivatives are evaluated at $(t+\eta,u+\eta)$.
For a prescribed error, first choose $K$ to make the geometric tails
small, and then choose $\gamma$ small enough for the finitely many terms
$d\le K$. This yields an error tending uniformly to zero as
$\gamma\downarrow0$.
For every $j$, the truncated sum is therefore at most
\[
 e^{-G_t}A(2u-t,e^{-G_u})+o(1).
\]
For $j\theta/\pi\geq\eta$, all fixed $d\leq K$ are present for
small $\gamma$, and the same expression is also a lower approximation.

At the shifted point, \eqref{upper:eq:normalized-hj} gives
$e^{G_t}=A(2u-t+\eta,e^{-G_u})$.
The uniform continuity of $\log a_d$ implies
\[
 e^{-\omega_B(\eta)}
 \leq
 \frac{A(2u-t,e^{-G_u})}{A(2u-t+\eta,e^{-G_u})}
 \leq e^{\omega_B(\eta)}.
\]
All these ratios are bounded away from zero for fixed $\eta$.
The uniform additive errors can thus be absorbed into
$\varepsilon_{\gamma,\eta}$, proving the lemma.
\end{proof}

\subsection{Boundary estimates and convergence}

The zero initial row cannot be compared from below with a positive
exponential test array. We start at row $m=1$ instead. We also control
the strip $j\theta/\pi<\eta$, where truncation of the recurrence prevents
the two-sided one-step approximation.

On the first positive row the recurrence gives
\[
 c_B\leq\widetilde Q_{1,j}
       =\widetilde w_\theta(0,j)\leq C_B.
\]
The bounds follow from Lemma~\ref{upper:lem:uniform-local-weights}.
For the next lemma, retain $t=m\theta/\pi$, $u=j\theta/\pi$, and the
convention $0\log0=0$.

\begin{lemma}[Boundary estimates]\label{upper:lem:hj-boundary-estimates}
Uniformly for $1\le m\le N$ and $0\le j\le n$,
\begin{align}
 -C\left(u+\frac{\theta}{\pi}\right)
 &\leq
 \frac{\gamma^2}{4}\log\widetilde Q_{m,j}
 -\int_0^t\log a_0(-x)\,dx
 \notag\\
 &\leq C\left(u+\frac{\theta}{\pi}\right)
 +(t+u)\log(t+u)-t\log t-u\log u,
 \label{upper:eq:hj-strip-bound}
\end{align}
\end{lemma}

\begin{proof}
Expand the recurrence into integer choices
$d_1+\cdots+d_m=j$.
At block $k$, the state differs from $-(k-1)\theta/\pi$ by at most
$2u$. If $d_k=0$, the exact formula for $\widetilde w_\theta(b,0)$
and smoothness of $V$ give
\[
 \log\widetilde w_\theta(b,0)=\log a_0(b)+O(\theta/\pi)
\]
uniformly in $b$. Thus the logarithmic difference from
$a_0(-(k-1)\theta/\pi)$ is at most $C(u+\theta/\pi)$.
If $d_k>0$, that difference is uniformly bounded, and there are
at most $j$ such blocks.

It follows that the scaled logarithmic weight of every composition
differs from
$\frac{\theta}{\pi}\sum_{k=1}^m\log a_0(-(k-1)\theta/\pi)$
by at most $C(u+\theta/\pi)$.
The latter sum differs from its integral by $O(\theta/\pi)$.
Keeping one composition proves the lower bound.
For the upper bound, there are $\binom{m+j-1}{j}$ compositions, and
the binomial theorem gives
\[
 \frac{\theta}{\pi}\log\binom{m+j-1}{j}
 \leq (t+u)\log(t+u)-t\log t-u\log u.
\]
This proves \eqref{upper:eq:hj-strip-bound}.
\end{proof}

\begin{proposition}[Convergence by positive recursion]
\label{upper:prop:hj-convergence}
Along $N=2W/\gamma^2\in\mathbb N$ and $\gamma^2n\to q\in(0,2)$,
\[
 \frac{\gamma^2}{4}\log\widetilde Q_{N,n}
 \longrightarrow G(W/2,q/4),\qquad
 \frac{\gamma^2}{4}\log Q_{N,n}
 \longrightarrow F(W/2,q/4).
\]
\end{proposition}

\begin{proof}
Fix $\eta>0$ as in the one-step lemma. Write
$t=m\theta/\pi$ and $u=j\theta/\pi$.
The boundary values of $G$ and
Lemma~\ref{upper:lem:hj-boundary-estimates} give numbers
$\beta_\eta>0$, tending to zero with $\eta$, such that
\[
 \left|\frac{\gamma^2}{4}\log\widetilde Q_{m,j}
       -G(t+\eta,u+\eta)\right|\leq\beta_\eta
\]
on row $m=1$ and throughout $u<\eta$, for all sufficiently small
$\gamma$.
On the first row the scaled logarithm is uniformly small and
$G(t+\eta,u+\eta)$ approaches its zero boundary value.
On the strip both quantities approach
$\int_0^t\log a_0(-x)\,dx$: the error in
\eqref{upper:eq:hj-strip-bound} tends uniformly to zero with $\eta$.

Define
\[
 U^\pm_{m,j}=
 \exp\left\{\frac4{\gamma^2}
 \left[G(t+\eta,u+\eta)\pm\beta_\eta
 \pm t\{\omega_B(\eta)+\varepsilon_{\gamma,\eta}\}\right]\right\}.
\]
These arrays bracket $\widetilde Q$ on the first row and on the strip.
A row step changes $t$ by $\gamma^2/4$, so
Lemma~\ref{upper:lem:hj-consistency} makes $U^+$ a supersolution
of \eqref{upper:eq:normalized-recursion} and $U^-$ a subsolution
for $u\geq\eta$.
Induction on $m$, using only positivity of the weights, gives
\[
 U^-_{m,j}\leq\widetilde Q_{m,j}\leq U^+_{m,j}.
\]
No PDE comparison theorem is being used.

At $(N,n)$, first let $\gamma\downarrow0$ with $\eta$ fixed.
The upper and lower limits differ from $G(W/2+\eta,q/4+\eta)$
by at most $\beta_\eta+(W/2)\omega_B(\eta)$.
Letting $\eta\downarrow0$ proves convergence of the normalized array.
Finally,
\[
 \frac{\gamma^2}{4}\log Q_{N,n}
 =
 \frac{\gamma^2}{4}\log\widetilde Q_{N,n}
 +V\left(\frac{(2n-N)\theta}{\pi}\right)-V(0),
\]
and $V(0)=0$, $F(t,u)=G(t,u)+V(2u-t)$.
This restores the exact normalization and proves the proposition.
\end{proof}

\subsection{The limiting moment transform}\label{sec:upper-results}\label{sec:analytic}

\begin{lemma}[Moment-transform limit]\label{upper:lem:transform}
For $0<q<2$, $\Lambda_{\gamma,W}(q)$ converges to
\begin{equation}\label{upper:eq:moment-limit}
 \Lambda_W(q)=q\log\frac{q}{4\pi}-q
 +4F(W/2,q/4)-4\int_0^{W/2}\log Z(t)\,\dd t.
\end{equation}
This limit extends continuously to $\Lambda_W(0)=0$.
\end{lemma}

\begin{proof}
Stirling's formula and a Riemann sum give
\[
 \gamma^2\log\left\{n!
 \left(\frac{\theta^2}{\pi^2\sin\theta}\right)^n\right\}
 \longrightarrow q\log\frac{q}{4\pi}-q,\qquad
 \gamma^2\sum_{j=0}^{N-1}\log Z(j\theta/\pi)
 \longrightarrow4\int_0^{W/2}\log Z(t)\,dt.
\]
Combining these with \eqref{upper:eq:exact-moment} and
Proposition~\ref{upper:prop:hj-convergence} proves the displayed formula.
Initially this establishes the limit for integer orders with
$\gamma^2n\to q$. To obtain the real order $q/\gamma^2$, let
$n_-=\lfloor q/\gamma^2\rfloor$, $n_+=n_-+1$.
For small $\gamma$ both moments are finite and $n_->0$.
Monotonicity of $L^p$ norms gives
\[
 \frac{q/\gamma^2}{n_-}\gamma^2\log\E[A_\gamma^{n_-}]
 \leq\gamma^2\log\E[A_\gamma^{q/\gamma^2}]
 \leq\frac{q/\gamma^2}{n_+}\gamma^2\log\E[A_\gamma^{n_+}].
\]
Both outer limits equal \eqref{upper:eq:moment-limit}.

The boundary values of $G$ imply
$F(t,0)=\int_0^t\log Z(s)\,ds$ and continuity as $u\downarrow0$.
Together with $q\log q\to0$, this proves
$\Lambda_W(q)\to0$ as $q\downarrow0$. We set $\Lambda_W(0)=0$.

\end{proof}

\subsection{A scalar formula for the derivative}

Let $C=C(W/2,q/4)$ be the unique parameter in
Lemma~\ref{upper:lem:hj-candidate}. Its endpoint equation is
\[
 \frac{q-W}{2}
 =
 -\frac{\arccos(C\cos(\pi W/2))}{\pi}
 \frac{C\sin(\pi W/2)}{\sqrt{1-C^2\cos^2(\pi W/2)}}.
\]
Set
\[
 y=\frac1\pi\arccos(C\cos(\pi W/2)),\qquad
 z=\frac{y^2(1-C^2)}{1-C^2\cos^2(\pi W/2)}.
\]
\begin{lemma}[Moment slopes]\label{upper:lem:slopes}
The function $\Lambda_W$ is real analytic on $(0,2)$, and
\begin{equation}\label{upper:eq:lambda-derivative}
 \Lambda_W'(q)=\log\frac{q}{4\pi z}.
\end{equation}
Its derivative increases strictly from $\log a_*(W)$ to $+\infty$.
\end{lemma}

\begin{proof}
Analyticity of the characteristic integral proves that
$\Lambda_W$ is real analytic for $0<q<2$.
Using $F_u=-\log z$ in \eqref{upper:eq:moment-limit} gives
\eqref{upper:eq:lambda-derivative}.
Together with $\Lambda_W(0)=0$, this determines the transform.

As $q\downarrow0$, $C\uparrow1$.
The endpoint derivative computed in the construction yields
\[
 \frac{q}{1-C}\longrightarrow
 \frac{\pi W-\sin(\pi W)}{\pi\sin^2(\pi W/2)},\qquad
 \frac{z}{1-C}\longrightarrow
 \frac{W^2}{2\sin^2(\pi W/2)}.
\]
Therefore
\[
 \Lambda_W'(0+)=\log a_*(W),\qquad
 a_*(W)=\frac{\pi W-\sin(\pi W)}{2\pi^2W^2}.
\]
As $q\uparrow2$, $C\downarrow-1$,
$y\to1-W/2>0$ and $z\to0$, so $\Lambda_W'(2-)=+\infty$.
Throughout, $\theta$ retains its meaning $\pi\gamma^2/4$.

H\"older's inequality gives convexity of the finite moment transforms,
and hence of their finite limit.
Thus $\Lambda_W'$ is nondecreasing. It is strictly increasing:
equal values at two different points would force it to be constant
on an interval, hence everywhere by analyticity, contradicting its
divergence at $q=2$.
Consequently, for every $a>a_*(W)$ there is a unique $q(a)\in(0,2)$
with $\Lambda_W'(q(a))=\log a$.

\end{proof}

\section{Convergence of the Laplace recursion}
\label{lower:sec:convergence}\label{part:lower}

Fix $s>0$; the case $s=0$ is immediate. We prove
$\gamma^2\log f_N(0)\to4\mathcal G(W/2,0)$.
The upper bound uses the one-step comparison on a bounded interval and
an estimate for paths leaving that interval. The lower bound uses a
change of measure: with probability tending to one under the new law,
the path follows the characteristic and its weighted likelihood has the
desired exponential rate.

\subsection{Exponential control of spatial exits}

Recall that
\[
    t_n=\frac{(n-1)\theta}{\pi}
       =\frac{(n-1)\gamma^2}{4}
\]
and that, when the \(n\)-th segment starts from \(x\),
\[
    X_{v,x}^{(n)}
    =
    x+\frac{\sqrt{2}\theta}{\pi}B_v
    -\frac{2t_n\theta}{\pi}v
    =
    x+\frac{\gamma^2}{2\sqrt{2}}B_v
    -\frac{\gamma^2t_n}{2}v.
\]
Choose \(p>0\) such that
\[
    \frac{N\theta}{\pi}+p<1.
\]
Concatenate the independent segments in the order $N,N-1,\ldots,1$,
starting at zero; $X$ denotes this continuous concatenated path.

\begin{lemma}[Spatial exit bound]\label{lower:lem:exit}
There is $C_p<\infty$, independent of $\gamma$, such that
\begin{equation}\label{lower:eq:exit}
 \mathbb P\left(
 \sup_{\text{concatenated process}}|X|>M
 \right)
 \leq
 \exp\left\{
 -\frac{4pM}{\gamma^2}+N\log C_p
 \right\}.
\end{equation}
\end{lemma}

\begin{proof}
Conditionally on \(V_n\), the reflection principle and the elementary
bound for the maximum of a Brownian motion give
\begin{align*}
 &\mathbb E\left[
 \left.
 \exp\left\{
 \frac{4p}{\gamma^2}
 \sup_{0\leq v\leq V_n}
 \left|X_{v,x}^{(n)}-x\right|
 \right\}
 \right|V_n
 \right] \\
 &\qquad\leq
 4\exp\left\{(p^2+2pt_n)V_n\right\}.
\end{align*}
Using
\[
    \mathbb E[e^{\lambda V_n}]
    =
    \frac{H(t_n^2+\lambda)}{H(t_n^2)},
\]
we obtain
\begin{equation}\label{lower:eq:oneblockexitmoment}
 \mathbb E\exp\left\{
 \frac{4p}{\gamma^2}
 \sup_{0\leq v\leq V_n}
 \left|X_{v,x}^{(n)}-x\right|
 \right\}
 \leq
 4\frac{H((t_n+p)^2)}{H(t_n^2)}.
\end{equation}
The right-hand side is bounded uniformly in \(n\leq N\), since
\(t_n+p<1\).  Thus there is a finite constant \(C_p\), independent of
\(n\) and \(\gamma\), such that the right-hand side of
\eqref{lower:eq:oneblockexitmoment} is at most \(C_p\).

The actual concatenated process traverses the segments in the order
\(N,N-1,\ldots,1\).  Its distance from the initial point is bounded by
the sum of the oscillations of the individual segments.  Since the
segments are independent, exponential Markov's inequality gives
\eqref{lower:eq:exit}.
Since \(\theta=\pi\gamma^2/4\), this implies
\begin{equation}\label{lower:eq:exitrate}
 \limsup_{\gamma\to0}
 \gamma^2\log
 \mathbb P\left(
 \sup_{\text{concatenated process}}|X|>M
 \right)
 \leq
 -4pM+\frac{4N\theta}{\pi}\log C_p.
\end{equation}
\end{proof}

The area factor is bounded by one, so the same estimate controls the
contribution of the exit event to the negative Laplace transform.

\subsection{Compact-domain upper comparison}

Fix \(M<\infty\) and let \(D=(-M,M)\).  When the \(n\)-th segment starts
from \(x\in D\), write
\[
 X_{v,x}^{(n)}
 =
 x+\frac{\sqrt{2}\theta}{\pi}B_v
 -\frac{2t_n\theta}{\pi}v
 =
 x+\frac{\gamma^2}{2\sqrt{2}}B_v
 -\frac{\gamma^2t_n}{2}v,
 \qquad
 t_n=\frac{(n-1)\theta}{\pi}.
\]
Let
\[
 \tau_D^{(n,x)}
 =
 \inf\{v\geq0:X_{v,x}^{(n)}\notin D\}.
\]
The killed version of the operator \(\mathcal K_n\) from Section~\ref{lower:sec:clock} is
\begin{align}
 (\mathcal K_n^D\varphi)(x)
 &=
 \mathbb E\left[
 \exp\left\{
 -\frac{s\gamma^2}{16\sin\theta}
 \int_0^{V_n}e^{2X_{v,x}^{(n)}}\dd v
 \right\}
 \varphi(X_{V_n,x}^{(n)});
 V_n<\tau_D^{(n,x)}
 \right].
 \label{lower:eq:killedoperator}
\end{align}
Thus \(\mathcal K_n^D\) is exactly \(\mathcal K_n\), with the additional
condition that the \(n\)-th segment remain in \(D\).

For a deterministic clock value \(r\), define
\begin{align}
 (T_{n,\gamma}^D(r)\varphi)(x)
 &=
 \mathbb E\left[
 \exp\left\{
 -\frac{s\gamma^2}{16\sin\theta}
 \int_0^r e^{2X_{v,x}^{(n)}}\dd v
 \right\}
 \varphi(X_{r,x}^{(n)});
 r<\tau_D^{(n,x)}
 \right].
 \label{lower:eq:killedsemigroup}
\end{align}
Since \(V_n\) is independent of the Brownian motion,
\[
 \mathcal K_n^D\varphi
 =
 \mathbb E\bigl[T_{n,\gamma}^D(V_n)\varphi\bigr].
\]
Both \(T_{n,\gamma}^D(r)\) and \(\mathcal K_n^D\) preserve pointwise
inequalities. Define operator
\[
 \mathcal A_{n,\gamma}
 =
 \frac{\gamma^4}{16}\partial_x^2
 -\frac{\gamma^2t_n}{2}\partial_x
 -\frac{s\gamma^2}{16\sin\theta}e^{2x}.
\]

By Feynman-Kac, $u(r,x) = (T_{n,\gamma}^D(r)\varphi)(x)$ solves PDE
\[\begin{cases}
    \partial_r u(r,x) = \mathcal A_{n,\gamma} u(r,x),\qquad r > 0, x\in D\\
    u(r,x) = 0, r > 0,\qquad x\in \partial D\\
    u(0,x) = \varphi(x),\qquad x\in D
\end{cases}
\]

For \(0\leq t\leq N\theta/\pi\), set
\[
 g_{\gamma,t}(x)
 =
 \exp\left\{\frac{4}{\gamma^2}\mathcal G(t,x)\right\}.
\]

\begin{lemma}[Uniform one-step supersolution estimate]
\label{lower:lem:consistency}
For every fixed \(M<\infty\), there exists
\(\eta_\gamma(M)\to0\) such that, for every \(1\leq n\leq N\) and every
\(x\in D\),
\begin{equation}\label{lower:eq:onestep}
 \mathcal K_n^Dg_{\gamma,t_n}(x)
 \leq
 e^{\eta_\gamma(M)}g_{\gamma,t_{n+1}}(x).
\end{equation}
\end{lemma}

\begin{proof}

We separate large clock values from bounded clock values.
A differential inequality controls the large-clock contribution.
For bounded clocks, an explicit change of measure controls the spatial
variation. Averaging and using the HJ equation gives the one-step bound.

First, we will construct a supersolution for PDE 
\[\begin{cases}
    \partial_r u(r,x) = \mathcal A_{n,\gamma} u(r,x),\qquad r > 0, x\in D\\
    u(r,x) = 0, \qquad r > 0, x\in \partial D\\
    u(0,x) = g_{\gamma,t_n}(x) ,\qquad x\in D
\end{cases}.
\]

Here, $u = T^D_{n,\gamma}(r)g_{\gamma, t_n}(x)$. Direct differentiation gives
\begin{align}
 \frac{\mathcal A_{n,\gamma}g_{\gamma,t_n}(x)}
      {g_{\gamma,t_n}(x)}
 &=
 (\mathcal G_x(t_n,x)-t_n)^2
 -\frac{s}{4\pi}e^{2x}
 -t_n^2+r_\gamma(t_n,x),
 \label{lower:eq:generatorconsistency}
\end{align}
where
\[
 r_\gamma(t,x)
 =
 \frac{\gamma^2}{4}\mathcal G_{xx}(t,x)
 -
 \left(
 \frac{s\gamma^2}{16\sin\theta}
 -\frac{s}{4\pi}
 \right)e^{2x}.
\]
Since \(\theta=\pi\gamma^2/4\), smoothness of \(\mathcal G\) gives
\[
 \sup_{\substack{0\leq t\leq N\theta/\pi\\x\in D}}
 |r_\gamma(t,x)|
 =O_M(\gamma^2).
\]

By the strict-domain condition and compactness, there exists
\(\delta_M>0\) such that
\[
 \sup_{\substack{0\leq t\leq N\theta/\pi\\x\in D}}
 \left\{
 (\mathcal G_x(t,x)-t)^2
 -\frac{s}{4\pi}e^{2x}
 \right\}
 \leq1-\delta_M.
\]
Consequently, for all sufficiently small \(\gamma\),
\[
 \mathcal A_{n,\gamma}g_{\gamma,t_n}(x)
 \leq
 \left(1-t_n^2-\frac{\delta_M}{2}\right)
 g_{\gamma,t_n}(x),
 \qquad x\in D.
\]
Set
\[
 v(r,x)
 =
 \exp\left\{
 \left(1-t_n^2-\frac{\delta_M}{2}\right)r
 \right\}g_{\gamma,t_n}(x).
\]
Then $v(r,x)$ is a supersolution, because
\begin{align*}
 \partial_rv-\mathcal A_{n,\gamma}v
 &=
 e^{(1-t_n^2-\delta_M/2)r}
 \left[
 \left(1-t_n^2-\frac{\delta_M}{2}\right)
 g_{\gamma,t_n}
 -
 \mathcal A_{n,\gamma}g_{\gamma,t_n}
 \right]
 \geq0.
\end{align*}
Moreover, \(u(0,\cdot)=v(0,\cdot)\) in \(D\), while
\(u=0\leq v\) on \(\partial D\).  The parabolic maximum principle
therefore gives
\begin{equation}\label{lower:eq:semigrouptailbound}
 T_{n,\gamma}^D(r)g_{\gamma,t_n}(x)
 \leq
 \exp\left\{
 \left(1-t_n^2-\frac{\delta_M}{2}\right)r
 \right\}g_{\gamma,t_n}(x).
\end{equation}
The above computation automatically yields a bound for $\mathcal K^D_n g_{\gamma, t_n}/g_{\gamma, t_{n+1}}$ when $V_n$ is large: 
\begin{align}
 &\frac{1}{g_{\gamma,t_n}(x)}
 \mathbb E\left[
 T_{n,\gamma}^D(V_n)g_{\gamma,t_n}(x);
 V_n>K
 \right] \notag\\
 &\qquad\leq
 e^{-\delta_MK/4}
 \mathbb E\left[
 e^{(1-t_n^2-\delta_M/4)V_n}
 \right] \notag\\
 &\qquad=
 e^{-\delta_MK/4}
 \frac{H(1-\delta_M/4)}{H(t_n^2)}
 \leq
 C_Me^{-\delta_MK/4}.
 \label{lower:eq:clocktailcomparison}
\end{align}

We next consider deterministic clock values \(0\leq r\leq K\).
Under the original probability measure, the process in the \(n\)-th
block satisfies
\[
 \dd X_v
 =
 \frac{\gamma^2}{2\sqrt{2}}\dd B_v
 -
 \frac{\gamma^2t_n}{2}\dd v,
 \qquad X_0=x.
\]
For \(v\leq\tau_D^{(n,x)}\), define
\begin{align}
 Z_v
 &=
 \exp\left\{
 \int_0^v
 \frac{\gamma^2}{2\sqrt{2}}\,
 \partial_x\log g_{\gamma,t_n}(X_a)\,\dd B_a
 \right. \notag\\
 &\hspace{27mm}\left.
 -
 \frac{\gamma^4}{16}
 \int_0^v
 \bigl(\partial_x\log g_{\gamma,t_n}(X_a)\bigr)^2\dd a
 \right\}.
 \label{lower:eq:girsanovdensity}
\end{align}
On the compact interval \(D\), the derivative
\(\partial_x\log g_{\gamma,t_n}\) is bounded.  Thus the stopped
stochastic exponential
\(Z_{v\wedge\tau_D^{(n,x)}}\) satisfies Novikov's condition and is a
martingale.

We define \(\widehat{\mathbb P}_x\) by
\[
 \left.
 \frac{\dd\widehat{\mathbb P}_x}{\dd\mathbb P_x}
 \right|_{\mathcal F_{r\wedge\tau_D^{(n,x)}}}
 =
 Z_{r\wedge\tau_D^{(n,x)}},
\]
and write \(\widehat{\mathbb E}_x\) for expectation under
\(\widehat{\mathbb P}_x\).

We now derive the change-of-measure identity.  It\^o's formula applied
to \(\log g_{\gamma,t_n}(X_v)\) gives
\begin{align}
 \dd\log g_{\gamma,t_n}(X_v)
 &=
 \frac{\gamma^2}{2\sqrt{2}}\,
 \partial_x\log g_{\gamma,t_n}(X_v)\,\dd B_v
 \notag\\
 &\quad+
 \left[
 -\frac{\gamma^2t_n}{2}
 \partial_x\log g_{\gamma,t_n}(X_v)
 +
 \frac{\gamma^4}{16}
 \partial_{xx}\log g_{\gamma,t_n}(X_v)
 \right]\dd v.
 \label{lower:eq:itologg}
\end{align}
On the other hand, the definition of \(Z_v\) gives
\begin{align}
 \dd\log Z_v
 &=
 \frac{\gamma^2}{2\sqrt{2}}\,
 \partial_x\log g_{\gamma,t_n}(X_v)\,\dd B_v
 \notag\\
 &\quad-
 \frac{\gamma^4}{16}
 \bigl(\partial_x\log g_{\gamma,t_n}(X_v)\bigr)^2\dd v.
 \label{lower:eq:itodensity}
\end{align}
Subtracting \eqref{lower:eq:itodensity} from
\eqref{lower:eq:itologg}, we obtain
\begin{align}
 \dd\log g_{\gamma,t_n}(X_v)-\dd\log Z_v
 &=
 \left[
 -\frac{\gamma^2t_n}{2}
 \partial_x\log g_{\gamma,t_n}(X_v)
 \right. \notag\\
 &\qquad\left.
 +
 \frac{\gamma^4}{16}
 \partial_{xx}\log g_{\gamma,t_n}(X_v)
 +
 \frac{\gamma^4}{16}
 \bigl(\partial_x\log g_{\gamma,t_n}(X_v)\bigr)^2
 \right]\dd v.
 \label{lower:eq:subtractito}
\end{align}
For every positive smooth function \(g\),
\[
 \frac{\partial_{xx}g}{g}
 =
 \partial_{xx}\log g
 +
 \bigl(\partial_x\log g\bigr)^2.
\]
Therefore the right-hand side of
\eqref{lower:eq:subtractito}, after subtracting the Feynman--Kac
potential, is precisely
\[
 \frac{\mathcal A_{n,\gamma}g_{\gamma,t_n}(X_v)}
      {g_{\gamma,t_n}(X_v)}\dd v,
\]
because
\begin{align*}
 \frac{\mathcal A_{n,\gamma}g_{\gamma,t_n}}
      {g_{\gamma,t_n}}
 &=
 -\frac{\gamma^2t_n}{2}
 \partial_x\log g_{\gamma,t_n}
 +
 \frac{\gamma^4}{16}
 \left[
 \partial_{xx}\log g_{\gamma,t_n}
 +
 \bigl(\partial_x\log g_{\gamma,t_n}\bigr)^2
 \right] \\
 &\qquad
 -
 \frac{s\gamma^2}{16\sin\theta}e^{2x}.
\end{align*}
Integrating up to time \(r\), on the event
\(\{r<\tau_D^{(n,x)}\}\), gives
\begin{align}
 &\exp\left\{
 -\frac{s\gamma^2}{16\sin\theta}
 \int_0^r e^{2X_v}\dd v
 \right\}
 \frac{g_{\gamma,t_n}(X_r)}
      {g_{\gamma,t_n}(x)}
 \notag\\
 &\qquad=
 Z_r
 \exp\left\{
 \int_0^r
 \frac{\mathcal A_{n,\gamma}g_{\gamma,t_n}(X_v)}
      {g_{\gamma,t_n}(X_v)}
 \dd v
 \right\}.
 \label{lower:eq:pathwisegroundstate}
\end{align}
Taking expectation under \(\mathbb P_x\) and using
\(\dd\widehat{\mathbb P}_x=Z_r\,\dd\mathbb P_x\) on the event
\(\{r<\tau_D^{(n,x)}\}\), we conclude that
\begin{align}
 \frac{T_{n,\gamma}^D(r)g_{\gamma,t_n}(x)}
      {g_{\gamma,t_n}(x)}
 &=
 \widehat{\mathbb E}_x\left[
 \exp\left\{
 \int_0^r
 \frac{\mathcal A_{n,\gamma}g_{\gamma,t_n}(Y_v)}
      {g_{\gamma,t_n}(Y_v)}
 \dd v
 \right\};
 r<\tau_D^{(n,x)}
 \right].
 \label{lower:eq:groundstate}
\end{align}

For completeness, we also identify the process under the new measure.
Girsanov's theorem shows that
\[
 \widehat B_v
 =
 B_v-
 \int_0^v
 \frac{\gamma^2}{2\sqrt{2}}\,
 \partial_x\log g_{\gamma,t_n}(X_a)\,\dd a
\]
is a Brownian motion under \(\widehat{\mathbb P}_x\), up to the exit
time.  Hence
\begin{align*}
 \dd X_v
 &=
 \frac{\gamma^2}{2\sqrt{2}}\dd\widehat B_v
 +
 \left[
 -\frac{\gamma^2t_n}{2}
 +
 \frac{\gamma^4}{8}
 \partial_x\log g_{\gamma,t_n}(X_v)
 \right]\dd v.
\end{align*}
Since
\[
 \partial_x\log g_{\gamma,t_n}(x)
 =
 \frac{4}{\gamma^2}\mathcal G_x(t_n,x),
\]
the transformed coordinate process, denoted by \(Y\), satisfies
\begin{equation}\label{lower:eq:transformedprocess}
 \dd Y_v
 =
 \frac{\gamma^2}{2\sqrt{2}}\dd\widehat B_v
 +
 \frac{\gamma^2}{2}
 \bigl(\mathcal G_x(t_n,Y_v)-t_n\bigr)\dd v,
 \qquad Y_0=x,
\end{equation}
up to its first exit from \(D\).
Extend the drift in \eqref{lower:eq:transformedprocess} and the
function appearing in \eqref{lower:eq:generatorconsistency} from \(D\)
to bounded Lipschitz functions on the real line, with bounds uniform in
\(n\).  The extended process agrees with \(Y\) before its first exit
from \(D\).  Since the integrand in
\eqref{lower:eq:groundstate} is positive after exponentiation, dropping
the exit indicator gives an upper bound.

For \(0\leq r\leq K\), the extended process satisfies
\[
 \sup_{0\leq v\leq r}|Y_v-x|
 \leq
 C_M\gamma^2K
 +
 \frac{\gamma^2}{2\sqrt{2}}
 \sup_{0\leq v\leq K}|\widehat B_v|.
\]
The Lipschitz bounds and the Gaussian exponential moments of the
Brownian maximum therefore imply, uniformly in \(n\), \(x\in D\), and
\(0\leq r\leq K\),
\begin{align}
 \frac{T_{n,\gamma}^D(r)g_{\gamma,t_n}(x)}
      {g_{\gamma,t_n}(x)}
 &\leq
 \exp\left\{
 r\left[
 (\mathcal G_x(t_n,x)-t_n)^2
 -\frac{s}{4\pi}e^{2x}
 -t_n^2
 \right]
 +\varepsilon_{\gamma,M,K}
 \right\},
 \label{lower:eq:boundedclockfixed}
\end{align}
where \(\varepsilon_{\gamma,M,K}\to0\) as \(\gamma\to0\), for fixed
\(M\) and \(K\).

Averaging over \(V_n\leq K\) and using
\[
 \mathbb E[e^{\lambda V_n}]
 =
 \frac{H(t_n^2+\lambda)}{H(t_n^2)}
\]
gives
\begin{align}
 &\frac{1}{g_{\gamma,t_n}(x)}
 \mathbb E\left[
 T_{n,\gamma}^D(V_n)g_{\gamma,t_n}(x);
 V_n\leq K
 \right] \notag\\
 &\qquad\leq
 e^{\varepsilon_{\gamma,M,K}}
 \frac{
 H\left(
 (\mathcal G_x(t_n,x)-t_n)^2
 -\frac{s}{4\pi}e^{2x}
 \right)}
 {H(t_n^2)}.
 \label{lower:eq:boundedclockcomparison}
\end{align}
On the compact rectangle, the last ratio is bounded above and bounded
away from zero.  Combining
\eqref{lower:eq:clocktailcomparison} and
\eqref{lower:eq:boundedclockcomparison}, then first sending
\(\gamma\to0\) and afterwards \(K\to\infty\), yields
\begin{align}
 \limsup_{\gamma\to0}
 \max_{1\leq n\leq N}\sup_{x\in D}
 \log\left[
 \frac{\mathcal K_n^Dg_{\gamma,t_n}(x)}
      {g_{\gamma,t_n}(x)}
 \frac{H(t_n^2)}
 {H\left(
 (\mathcal G_x(t_n,x)-t_n)^2
 -\frac{s}{4\pi}e^{2x}
 \right)}
 \right]
 \leq0.
 \label{lower:eq:spatialconsistency}
\end{align}

Finally,
\[
 t_{n+1}-t_n=\frac{\theta}{\pi}=\frac{\gamma^2}{4}.
\]
Taylor's theorem gives, uniformly in \(n\) and \(x\in D\),
\begin{align*}
 \frac{g_{\gamma,t_{n+1}}(x)}
      {g_{\gamma,t_n}(x)}
 &=
 \exp\left\{
 \mathcal G_t(t_n,x)+O_M(\gamma^2)
 \right\} \\
 &=
 \frac{
 H\left(
 (\mathcal G_x(t_n,x)-t_n)^2
 -\frac{s}{4\pi}e^{2x}
 \right)}
 {H(t_n^2)}
 e^{O_M(\gamma^2)},
\end{align*}
where the Hamilton--Jacobi equation was used in the second equality.
Together with \eqref{lower:eq:spatialconsistency}, this proves
\eqref{lower:eq:onestep}.
\end{proof}

Define the killed recursion by
\[
 f_0^D(x)=1,
 \qquad
 f_n^D=\mathcal K_n^Df_{n-1}^D,
 \qquad 1\leq n\leq N.
\]
\begin{proposition}[Laplace upper bound]\label{lower:prop:upper}
\[
 \limsup_{\gamma\to0}\gamma^2\log f_N(0)\le4\mathcal G(W/2,0).
\]
\end{proposition}

\begin{proof}
Because \(t_1=0\) and \(\mathcal G(0,x)=0\),
\[
 f_0^D(x)=g_{\gamma,t_1}(x)=1.
\]
Positivity of \(\mathcal K_n^D\) and
Lemma~\ref{lower:lem:consistency} give inductively
\[
 f_n^D(x)
 \leq
 e^{n\eta_\gamma(M)}g_{\gamma,t_{n+1}}(x).
\]
Since \(t_{N+1}=N\theta/\pi\), we obtain
\begin{equation}\label{lower:eq:killedupper}
 f_N^D(0)
 \leq
 \exp\left\{
 N\eta_\gamma(M)
 +
 \frac{4}{\gamma^2}
 \mathcal G\left(\frac{N\theta}{\pi},0\right)
 \right\}.
\end{equation}
Moreover,
\[
 \gamma^2N\eta_\gamma(M)
 =
 \frac{4N\theta}{\pi}\eta_\gamma(M)
 \longrightarrow0.
\]

The killed recursion is the contribution from paths for which every
Brownian segment remains in \(D\).  Since the area factor belongs to
\([0,1]\),
\[
 0\leq f_N(0)-f_N^D(0)
 \leq
 \mathbb P\left(
 \sup_{\text{concatenated process}}|X|>M
 \right).
\]
Combining \eqref{lower:eq:killedupper} with the exit estimate from the
Lemma~\ref{lower:lem:exit}, and then choosing \(M\) sufficiently large, gives
\begin{equation}\label{lower:eq:upper}
 \limsup_{\gamma\to0}
 \left\{
 \gamma^2\log f_N(0)
 -
 4\mathcal G\left(\frac{N\theta}{\pi},0\right)
 \right\}
 \leq0.
\end{equation}

\end{proof}

\subsection{Lower bound by change of measure}

The goal is to prove a lower bound for
\(f_N(0)=\mathbb E[\exp(-sA_\gamma/\gamma^2)]\).
We will construct a probability measure
\(\widehat{\mathbb P}_\gamma\) under which the log-length process
concentrates around a suitable deterministic trajectory.  If
\(L_\gamma=\dd\widehat{\mathbb P}_\gamma/\dd\mathbb P\), then
\[
 f_N(0)
 =
 \widehat{\mathbb E}_\gamma
 \left[
 e^{-sA_\gamma/\gamma^2}L_\gamma^{-1}
 \right].
\]
Set \(Z_\gamma=e^{-sA_\gamma/\gamma^2}L_\gamma^{-1}\).
We will prove that
\[
 \frac{\gamma^2}{4}\log Z_\gamma
 -
 \mathcal G\left(\frac{N\theta}{\pi},0\right)
 \longrightarrow0
 \quad\text{in }\widehat{\mathbb P}_\gamma\text{-probability}.
\]
Consequently, for every \(\varepsilon>0\), with $\widehat{\mathbb P}_\gamma$-probability tending to one,
\[
 Z_\gamma\geq
 \exp\left\{
 \frac4{\gamma^2}
 \left[
 \mathcal G\left(\frac{N\theta}{\pi},0\right)-\varepsilon
 \right]
 \right\}.
\]
Restricting the expectation under $\widehat{\mathbb P}_\gamma$ to this event will give the
desired lower bound.

Throughout, \(\theta=\pi\gamma^2/4\), and \(N\theta/\pi<1\) is fixed
along the sequence under consideration.

\paragraph{The target trajectory and the block indexing.}

Let \(x_*(t)\), \(0\leq t\leq N\theta/\pi\), be the characteristic
satisfying \(x_*(N\theta/\pi)=0\).  Define
\[
 P(t)=\mathcal G_x(t,x_*(t)),
 \qquad
 v_*(t)
 =
 (P(t)-t)^2-\frac{s}{4\pi}e^{2x_*(t)},
 \qquad
 h(v)=\frac{H'(v)}{H(v)}.
\]
The characteristic equation is
\begin{equation}\label{lower:eq:target-characteristic}
 x_*'(t)=-2(P(t)-t)h(v_*(t)).
\end{equation}
These functions are smooth and bounded on
\([0,N\theta/\pi]\).  Moreover, the strict-domain condition gives
a constant \(\delta>0\) such that
\(v_*(t)\leq1-\delta\) throughout this interval.

We keep the block index from Section~\ref{lower:sec:clock}:
\[
 t_n=\frac{(n-1)\theta}{\pi}
     =\frac{(n-1)\gamma^2}{4},
 \qquad 1\leq n\leq N+1.
\]
The physical process traverses blocks \(N,N-1,\ldots,1\).
Write \(Y_n\) for the starting position of block \(n\), and
\(Y_{n-1}\) for its endpoint.  Thus \(Y_N=0\), and
\[
 X_v^{(n)}
 =
 Y_n+\frac{\gamma^2}{2\sqrt2}B_v^{(n)}
 -\frac{\gamma^2t_n}{2}v,
 \qquad
 Y_{n-1}=X_{V_n}^{(n)}.
\]
The pairs \((B^{(n)},V_n)\) are independent across blocks, and
\(B^{(n)}\) is independent of \(V_n\).  Set
\[
 \Delta_n=Y_{n-1}-Y_n
 =
 \frac{\gamma^2}{2\sqrt2}B_{V_n}^{(n)}
 -\frac{\gamma^2t_n}{2}V_n.
\]
The target starting position of block \(n\) is
\(x_*(t_{n+1})\), while its target endpoint is \(x_*(t_n)\).
In particular, its target displacement is
\[
 x_*(t_n)-x_*(t_{n+1})
 =
 -\frac{\gamma^2}{4}x_*'(t_n)+O(\gamma^4).
\]

\paragraph{Construction of the change of measure.}

We first explain the choice of the one-block density.
For a block starting at \(x\), the area weight multiplied by the
predicted continuation value, relative to its value at \(x\), is
\[
 \exp\left\{
 -\frac{s\gamma^2}{16\sin\theta}
 \int_0^{V_n}e^{2X_v^{(n)}}\,\dd v
 +
 \frac4{\gamma^2}
 \bigl[
 \mathcal G(t_n,x+\Delta_n)-\mathcal G(t_n,x)
 \bigr]
 \right\}.
\]
Freezing the area integrand at \(x\), linearizing the last difference,
and using \(s\gamma^2/(16\sin\theta)\to s/(4\pi)\), suggests the weight
\[
 \exp\left\{
 \frac4{\gamma^2}\mathcal G_x(t_n,x)\Delta_n
 -\frac{s}{4\pi}e^{2x}V_n
 \right\}.
\]
For the lower bound, we freeze these coefficients at the deterministic
point \(x_*(t_n)\).  This motivates the following exact definition:
\begin{equation}\label{lower:eq:control-density}
 L_n
 =
 \exp\left\{
 \frac{4P(t_n)}{\gamma^2}\Delta_n
 -\frac{s}{4\pi}e^{2x_*(t_n)}V_n
 \right\}
 \frac{H(t_n^2)}{H(v_*(t_n))}.
\end{equation}
Set $L_\gamma=\prod_{n=1}^NL_n$. Once normalization is verified,
the new probability measure is defined by
\begin{equation}\label{lower:eq:control-law}
 \frac{\dd\widehat{\mathbb P}_\gamma}{\dd\mathbb P}=L_\gamma
\end{equation}

\begin{lemma}[Normalization of the change of measure]\label{lower:lem:tilted-law}
Each $L_n$ has expectation one. The density $L_\gamma$ defines a
probability measure under which the block pairs remain independent.
\end{lemma}

\begin{proof}
Conditionally on \(V_n=v\), the Gaussian moment formula gives
\[
 \mathbb E\left[
 \left.e^{4P(t_n)\Delta_n/\gamma^2}\right|V_n=v
 \right]
 =
 e^{(P(t_n)^2-2P(t_n)t_n)v}.
\]
Since
\[
 P(t_n)^2-2P(t_n)t_n
 -\frac{s}{4\pi}e^{2x_*(t_n)}
 =
 v_*(t_n)-t_n^2,
\]
the clock identity
\(\mathbb E e^{\lambda V_n}=H(t_n^2+\lambda)/H(t_n^2)\)
implies
\[
 \mathbb E\exp\left\{
 \frac{4P(t_n)}{\gamma^2}\Delta_n
 -\frac{s}{4\pi}e^{2x_*(t_n)}V_n
 \right\}
 =
 \frac{H(v_*(t_n))}{H(t_n^2)}.
\]
Thus \(\mathbb E L_n=1\).

Since each $L_n$ depends only on the independent pair $(B^{(n)},V_n)$,
the product has expectation one and block independence is preserved.

To identify this law explicitly, factor the density as
\begin{align*}
 L_n
 ={}&
 \exp\left\{
 \sqrt2P(t_n)B_{V_n}^{(n)}-P(t_n)^2V_n
 \right\}\\
 &\times
 \exp\left\{
 \bigl(v_*(t_n)-t_n^2\bigr)V_n
 \right\}
 \frac{H(t_n^2)}{H(v_*(t_n))}.
\end{align*}
For each fixed clock value, the first factor changes the Brownian drift
to \(\sqrt2P(t_n)\).  The second factor changes the clock distribution.
Under \(\widehat{\mathbb P}_\gamma\), we may therefore realize the blocks
using independent standard Brownian motions \(\widehat B^{(n)}\),
independent of the new clocks, such that
\begin{equation}\label{lower:eq:controlled-clock}
 \widehat{\mathbb E}_\gamma e^{\lambda V_n}
 =
 \frac{H(v_*(t_n)+\lambda)}{H(v_*(t_n))},
 \qquad \lambda<1-v_*(t_n),
\end{equation}
and
\begin{equation}\label{lower:eq:controlled-segment}
 X_v^{(n)}
 =
 Y_n+\frac{\gamma^2}{2\sqrt2}\widehat B_v^{(n)}
 +\frac{\gamma^2}{2}(P(t_n)-t_n)v.
\end{equation}
In particular,
\begin{align}
 \widehat{\mathbb E}_\gamma\Delta_n
 &=
 \frac{\gamma^2}{2}
 (P(t_n)-t_n)h(v_*(t_n)),
 \label{lower:eq:increment-moments}\\
 \operatorname{Var}_{\widehat{\mathbb P}_\gamma}(\Delta_n)
 &=
 \frac{\gamma^4}{8}h(v_*(t_n))
 +
 \frac{\gamma^4}{4}(P(t_n)-t_n)^2h'(v_*(t_n))
 \leq C\gamma^4.
 \notag
\end{align}
Here we used
\(\widehat{\mathbb E}_\gamma V_n=h(v_*(t_n))\) and
\(\operatorname{Var}_{\widehat{\mathbb P}_\gamma}(V_n)
=h'(v_*(t_n))\).

By \eqref{lower:eq:target-characteristic},
\[
 \widehat{\mathbb E}_\gamma\Delta_n
 =
 x_*(t_n)-x_*(t_{n+1})+O(\gamma^4).
\]
Thus the expected displacement under the new law agrees, to leading
order, with the displacement along the target characteristic.

\end{proof}

\begin{lemma}[Concentration around the characteristic]
\label{lower:lem:controlled-convergence}
Under $\widehat{\mathbb P}_\gamma$, in probability,
\begin{equation}\label{lower:eq:controlled-path-convergence}
 \max_{1\le n\le N}\sup_{0\le v\le V_n}
 |X_v^{(n)}-x_*(t_n)|\longrightarrow0.
\end{equation}
\end{lemma}

\begin{proof}
The strict margin \(v_*(t_n)\leq1-\delta\), compactness, and
\eqref{lower:eq:controlled-clock} imply that, for some
\(\lambda_0>0\),
\begin{equation}\label{lower:eq:uniform-controlled-clock}
 \sup_{\gamma,n}
 \widehat{\mathbb E}_\gamma e^{\lambda_0V_n}<\infty.
\end{equation}
In particular, all clock moments are uniformly bounded.

We first control the block endpoints.
Since \(Y_N=0\), we have
\(Y_n=\sum_{k=n+1}^N\Delta_k\).
Independence and Doob's maximal inequality, applied in the physical
order \(N,N-1,\ldots,1\), give
\[
 \widehat{\mathbb E}_\gamma
 \max_{0\leq n\leq N}
 \left|
 \sum_{k=n+1}^N
 \bigl(\Delta_k-\widehat{\mathbb E}_\gamma\Delta_k\bigr)
 \right|^2
 \leq
 4\sum_{k=1}^N
 \operatorname{Var}_{\widehat{\mathbb P}_\gamma}(\Delta_k)
 \leq CN\gamma^4=O(\gamma^2).
\]
On the other hand, the estimate for the mean displacement telescopes:
\[
 \max_{0\leq n\leq N}
 \left|
 \sum_{k=n+1}^N\widehat{\mathbb E}_\gamma\Delta_k
 -x_*(t_{n+1})
 \right|
 \leq CN\gamma^4=O(\gamma^2).
\]
We used \(x_*(t_{N+1})=0\).
It follows that
\[
 \max_{0\leq n\leq N}
 |Y_n-x_*(t_{n+1})|\longrightarrow0
 \quad\text{in probability}.
\]

To include segment interiors, fix a sufficiently large constant
\(C_0>0\).  The exponential clock bound implies
\[
 \widehat{\mathbb P}_\gamma
 \left(
 \max_{1\leq n\leq N}V_n>C_0\log(1/\gamma)
 \right)
 \leq CN\gamma^{\lambda_0C_0}\longrightarrow0.
\]
On the complementary event, the absolute drift displacement in
\eqref{lower:eq:controlled-segment} is at most
\(C\gamma^2\log(1/\gamma)\), which tends to zero.
For any fixed \(\varepsilon>0\), the reflection principle also gives
\[
 \widehat{\mathbb P}_\gamma
 \left(
 \max_{1\leq n\leq N}
 \sup_{v\leq C_0\log(1/\gamma)}
 \frac{\gamma^2}{2\sqrt2}
 |\widehat B_v^{(n)}|>\varepsilon
 \right)
 \leq
 4N\exp\left\{
 -\frac{4\varepsilon^2}
 {C_0\gamma^4\log(1/\gamma)}
 \right\}
 \longrightarrow0.
\]
Hence the maximum displacement within any one segment tends to zero.
Since
\(\sup_n|x_*(t_{n+1})-x_*(t_n)|=O(\gamma^2)\),
this proves \eqref{lower:eq:controlled-path-convergence}.

\end{proof}

Fix a bounded continuous function $a$ on $[0,W/2]$.
\begin{lemma}[Averaging the clocks]\label{lower:lem:clock-averaging}
Under $\widehat{\mathbb P}_\gamma$, in probability,
\begin{equation}\label{lower:eq:controlled-clock-lln}
 \frac{\gamma^2}{4}\sum_{n=1}^N
 a(t_n)[V_n-h(v_*(t_n))]\longrightarrow0.
\end{equation}
\end{lemma}

\begin{proof}
Independence and the uniform second-moment bound
give
\[
 \operatorname{Var}_{\widehat{\mathbb P}_\gamma}
 \left(
 \frac{\gamma^2}{4}\sum_{n=1}^N
 a(t_n)\bigl[V_n-h(v_*(t_n))\bigr]
 \right)
 \leq C\gamma^4N=O(\gamma^2).
\]
This proves \eqref{lower:eq:controlled-clock-lln}.  Taking \(a=1\)
and using a Riemann sum, we obtain
\[
 \frac{\gamma^2}{4}\sum_{n=1}^N V_n
 \longrightarrow
 \int_0^{N\theta/\pi}h(v_*(t))\,\dd t.
\]
In particular, the scaled total clock is bounded in probability.

\end{proof}

\begin{lemma}[Replacing area by its value along the characteristic]
\label{lower:lem:area-freezing}
Under $\widehat{\mathbb P}_\gamma$, in probability,
\begin{equation}\label{lower:eq:controlled-area-freezing}
 \frac{\gamma^2}{4}\sum_{n=1}^N
 \left[\frac{s}{4\pi}e^{2x_*(t_n)}V_n
 -\frac{s\gamma^2}{16\sin\theta}\int_0^{V_n}e^{2X_v^{(n)}}\,\dd v
 \right]\longrightarrow0.
\end{equation}
\end{lemma}

\begin{proof}
By Lemma~\ref{lower:lem:controlled-convergence},
with probability tending to one, all the paths
\(X_v^{(n)}\) and the characteristic lie in one fixed compact
interval.  On this event, \(e^{2x}\) is uniformly Lipschitz, and
\begin{align*}
 &\frac{\gamma^2}{4}\sum_{n=1}^N
 \left|
 \frac{s}{4\pi}e^{2x_*(t_n)}V_n
 -
 \frac{s\gamma^2}{16\sin\theta}
 \int_0^{V_n}e^{2X_v^{(n)}}\,\dd v
 \right|\\
 &\quad\leq
 C\left[
 \left|
 \frac{s\gamma^2}{16\sin\theta}-\frac{s}{4\pi}
 \right|
 +
 \max_{1\leq n\leq N}\sup_{v\leq V_n}
 |X_v^{(n)}-x_*(t_n)|
 \right]
 \frac{\gamma^2}{4}\sum_{n=1}^N V_n.
\end{align*}
The bracket tends to zero in probability, while the last factor is
bounded in probability.  This proves
\eqref{lower:eq:controlled-area-freezing}.
\end{proof}

\paragraph{The weighted likelihood along the characteristic.}

The exact Brownian-clock representation and
\eqref{lower:eq:control-law} give
\begin{equation}\label{lower:eq:controlled-change-measure}
 f_N(0)=\widehat{\mathbb E}_\gamma[Z_\gamma],
 \qquad
 Z_\gamma
 =
 \exp\left\{
 -\frac{s\gamma^2}{16\sin\theta}
 \sum_{n=1}^N\int_0^{V_n}e^{2X_v^{(n)}}\,\dd v
 \right\}
 \prod_{n=1}^NL_n^{-1}.
\end{equation}
\begin{lemma}[Weighted likelihood]\label{lower:lem:likelihood}
Under $\widehat{\mathbb P}_\gamma$, in probability,
\[
 \frac{\gamma^2}{4}\log Z_\gamma\longrightarrow\mathcal G(W/2,0).
\]
\end{lemma}

\begin{proof}
Taking logarithms gives
\begin{align}
 \frac{\gamma^2}{4}\log Z_\gamma
 ={}&
 \frac{\gamma^2}{4}\sum_{n=1}^N
 \log\frac{H(v_*(t_n))}{H(t_n^2)}
 -
 \sum_{n=1}^NP(t_n)\Delta_n
 \notag\\
 &+
 \frac{\gamma^2}{4}\sum_{n=1}^N
 \left[
 \frac{s}{4\pi}e^{2x_*(t_n)}V_n
 -
 \frac{s\gamma^2}{16\sin\theta}
 \int_0^{V_n}e^{2X_v^{(n)}}\,\dd v
 \right].
 \label{lower:eq:controlled-loglikelihood}
\end{align}
The first term is a deterministic Riemann sum:
\[
 \frac{\gamma^2}{4}\sum_{n=1}^N
 \log\frac{H(v_*(t_n))}{H(t_n^2)}
 \longrightarrow
 \int_0^{N\theta/\pi}
 \log\frac{H(v_*(t))}{H(t^2)}\,\dd t.
\]
For the second term, decompose
\[
 \sum_{n=1}^NP(t_n)\Delta_n
 =
 \sum_{n=1}^NP(t_n)\widehat{\mathbb E}_\gamma\Delta_n
 +
 \sum_{n=1}^NP(t_n)
 \bigl(\Delta_n-\widehat{\mathbb E}_\gamma\Delta_n\bigr).
\]
The centered sum has variance at most \(CN\gamma^4=O(\gamma^2)\),
so it converges to zero in \(L^2\).  The mean sum satisfies
\[
 \sum_{n=1}^NP(t_n)\widehat{\mathbb E}_\gamma\Delta_n
 =
 \frac{\gamma^2}{2}\sum_{n=1}^N
 P(t_n)(P(t_n)-t_n)h(v_*(t_n))
 \longrightarrow
 2\int_0^{N\theta/\pi}
 P(t)(P(t)-t)h(v_*(t))\,\dd t.
\]
The last line of \eqref{lower:eq:controlled-loglikelihood}
converges to zero by
\eqref{lower:eq:controlled-area-freezing}.  Therefore
\[
 \frac{\gamma^2}{4}\log Z_\gamma
 \longrightarrow
 \int_0^{N\theta/\pi}
 \left[
 \log\frac{H(v_*(t))}{H(t^2)}
 -
 2P(t)(P(t)-t)h(v_*(t))
 \right]\dd t
\]
in $\widehat{\mathbb P}_\gamma$-probability.

To identify the integral, differentiate \(\mathcal G\) along the
characteristic.  The Hamilton--Jacobi equation and
\eqref{lower:eq:target-characteristic} give
\begin{align*}
 \frac{\dd}{\dd t}\mathcal G(t,x_*(t))
 &=
 \mathcal G_t(t,x_*(t))
 +\mathcal G_x(t,x_*(t))x_*'(t)\\
 &=
 \log\frac{H(v_*(t))}{H(t^2)}
 -
 2P(t)(P(t)-t)h(v_*(t)).
\end{align*}
Since \(\mathcal G(0,x)=0\) and \(x_*(N\theta/\pi)=0\), we conclude
\begin{equation}\label{lower:eq:controlled-likelihood-limit}
 \frac{\gamma^2}{4}\log Z_\gamma
 -
 \mathcal G\left(\frac{N\theta}{\pi},0\right)
 \longrightarrow0
 \quad\text{in }\widehat{\mathbb P}_\gamma\text{-probability}.
\end{equation}

\end{proof}

\begin{proposition}[Laplace lower bound]\label{lower:prop:lower}
\[
 \liminf_{\gamma\to0}\gamma^2\log f_N(0)\ge4\mathcal G(W/2,0).
\]
\end{proposition}

\begin{proof}
For \(\varepsilon>0\), let
\[
 E_{\gamma,\varepsilon}
 =
 \left\{
 \frac{\gamma^2}{4}\log Z_\gamma
 \geq
 \mathcal G\left(\frac{N\theta}{\pi},0\right)-\varepsilon
 \right\}.
\]
Then
\(\widehat{\mathbb P}_\gamma(E_{\gamma,\varepsilon})\to1\).
Restricting the expectation in
\eqref{lower:eq:controlled-change-measure} to this event gives
\[
 f_N(0)
 \geq
 \exp\left\{
 \frac4{\gamma^2}
 \left[
 \mathcal G\left(\frac{N\theta}{\pi},0\right)-\varepsilon
 \right]
 \right\}
 \widehat{\mathbb P}_\gamma(E_{\gamma,\varepsilon}).
\]
Taking logarithms, multiplying by \(\gamma^2/4\), and then letting
\(\gamma\to0\) and \(\varepsilon\downarrow0\), we obtain
\[
 \liminf_{\gamma\to0}
 \left[
 \frac{\gamma^2}{4}\log f_N(0)
 -
 \mathcal G\left(\frac{N\theta}{\pi},0\right)
 \right]
 \geq0.
\]
\end{proof}

\begin{proposition}[Convergence of the Laplace transform]
\label{lower:prop:convergence}
The limit $\Psi_W(s)$ exists and equals $4\mathcal G(W/2,0)$.
\end{proposition}

\begin{proof}
Propositions~\ref{lower:prop:upper} and~\ref{lower:prop:lower} give
\begin{equation}\label{lower:eq:limit}
 \gamma^2\log f_N(0)
 -
 4\mathcal G\left(\frac{N\theta}{\pi},0\right)
 \longrightarrow0.
\end{equation}
\end{proof}

\subsection{The limiting Laplace transform}\label{sec:lower-results}
For $s>0$, let $C>1$ be a solution of
\begin{equation}\label{lower:eq:scalar}
 v=\chi^{-1}(C\cos(\pi W/2)),\qquad
 s=\frac4\pi(C^2-1)H(v)^2.
\end{equation}
Here $C\cos(\pi W/2)>-1$ is required; if $W>1$ this means
$C<-1/\cos(\pi W/2)$, while there is no upper restriction if $W\le1$.
Existence and uniqueness follow from Lemma~\ref{lower:lem:hj-candidate}.

\begin{theorem}[Negative-Laplace limit]\label{thm:lower-transform}
For every $s\ge0$, the finite limit
\[
 \Psi_W(s)=\lim_{\gamma\downarrow0}\Psi_{\gamma,W}(s)
\]
exists along $N=2W/\gamma^2\in\mathbb N$. For $s>0$,
\begin{equation}\label{lower:eq:derivative}
 \boxed{\Psi_W'(s)=\frac1s
 \left(W-\frac{2C\sin(\pi W/2)H(v)}\pi\right).}
\end{equation}
The right-hand side is negative. Together with $\Psi_W(0)=0$ it
determines $\Psi_W$ by a one-dimensional integral. In particular,
\[
 \Psi_W'(0+)=-\frac{\pi W-\sin(\pi W)}{2\pi^2W^2}.
\]
\end{theorem}

The parameter $v$ in \eqref{lower:eq:scalar} may be negative. This is a
hyperbolic branch of real functions, not a use of a positive exponential
moment of area. The strict condition $v<1$ will control the clock
transforms in the proof.

\begin{proof}
Proposition~\ref{lower:prop:convergence} gives
$\Psi_W(s)=4\mathcal G(W/2,0)$.
To display the dependence on $s$, write $\mathcal G_s$ in this proof.
The characteristic construction depends on $s$ and $x$ through
$s e^{2x}$, and hence
$\mathcal G_s(t,x)=\mathcal G_{s e^{2x}}(t,0)$.
Therefore
\[
 \Psi_W'(s)=\frac2s\,\partial_x\mathcal G_s(W/2,0).
\]
At the endpoint $x_C(W/2)=0$,
\eqref{lower:eq:charz} gives \eqref{lower:eq:scalar}, while
$\partial_x\mathcal G_s(W/2,0)=p_C(W/2)+W/2$ gives
\eqref{lower:eq:derivative}.

The same endpoint formulas give $v_C(t)\to t^2$ at $x=0$,
uniformly for $0\le t\le W/2$, as $s\downarrow0$.
Here the derivative of $-z_C(t)$ at $C=1$ extends to the positive value
$2/\pi^2$ at $t=0$, so the implicit endpoint solution is uniform on
this compact interval.
The HJ equation then gives $\partial_t\mathcal G_s(t,0)\to0$ uniformly.
Integration from $t=0$ proves $\Psi_W(s)\to0$, as required.

For the right derivative at zero, set $t=W/2$ and let
$z_C(t)=v_C(t)-p_C(t)^2$ as in the construction.
As $s\downarrow0$, $C\downarrow1$, $v_C(t)\to t^2$ and
$p_C(t)\to-t$. Differentiating the explicit formulas at $C=1$ gives
\[
 \left.\frac{d}{dC}[-z_C(t)]\right|_{C=1}
 =\frac{2t^2}{\sin^2(\pi t)},\qquad
 \left.\frac{d}{dC}p_C(t)\right|_{C=1}
 =\frac{\sin(\pi t)\cos(\pi t)/\pi-t}{\sin^2(\pi t)}.
\]
Since $s=4\pi[-z_C(t)]$, their ratio yields
$\Psi_W'(0+)=-[\pi W-\sin(\pi W)]/(2\pi^2W^2)$.
Finally, the characteristic equation gives $p_C'(t)<-1$ and
$p_C(0)=0$, hence $p_C(W/2)<-W/2$.
The formula $\Psi_W'(s)=2[p_C(W/2)+W/2]/s$ is therefore negative.
\end{proof}

Let $\vartheta=\pi W/2$ and let $(C,v)$ solve
\eqref{lower:eq:scalar}. Define
\[
 a_W(s)=-\Psi_W'(s)
 =\frac{2C\sin\vartheta\,H(v)/\pi-W}{s},\qquad s>0.
\]
\begin{lemma}[Laplace slopes]\label{lower:lem:slopes}
The function $a_W$ decreases continuously and strictly from $a_*(W)$
to zero, and maps $(0,\infty)$ onto $(0,a_*(W))$.
\end{lemma}

\begin{proof}
The characteristic formulas give $p<-W/2$, $p^2=v+s/(4\pi)$, and
$p=-C\sin\vartheta H(v)/\pi$. Consequently
\[
 0<a_W(s)=\frac{2\sqrt{v+s/(4\pi)}-W}{s}
 \le\frac{2\sqrt{1+s/(4\pi)}}s\longrightarrow0.
\]
The right derivative at zero gives $a_W(0+)=a_*(W)$.
The functions $\Psi_{\gamma,W}$ are convex, so their finite limit
$\Psi_W$ is convex. The scalar construction gives analyticity on
$(0,\infty)$. Hence $a_W$ is strictly decreasing: equality of its
values at two distinct points would make the analytic derivative
constant on an interval and then everywhere, contradicting these
endpoint limits. It follows that $a_W$ maps $(0,\infty)$ bijectively
onto $(0,a_*(W))$.

\end{proof}

\section{Removing the integer-block restriction}\label{sec:all-gamma}

We now allow arbitrary $\gamma\downarrow0$, with
$\theta=\pi\gamma^2/4$. The comparison uses $N$ and $N+1$ ordinary
blocks at the same value of $\gamma$. Their total weights approach
$W$, so their transforms have the same limits. We first justify this
moving-endpoint version of the convergence results, and then control
the boundary reweighting in the welding comparison.

For $0<w<2$, let $\mathbb{P}_{\gamma,w}$ be the normalized law of
$\mathcal M(w;1,\bullet)$, and let $A$ and $R$ denote its area and
right boundary length. The normalization is finite for sufficiently
small $\gamma$. We also use this notation for residual weights in
$(0,\gamma^2/2)$; finiteness in that case follows from
\cite[Lemma~2.18]{ConfWeldDisks}. Expectations are denoted by
$\E_{\gamma,w}$.

\subsection{Integer blocks with a moving terminal weight}

Consider any sequence of integers $N$ with
$N\gamma^2/2\to W\in(0,2)$.

\begin{lemma}[Moving endpoints]\label{lem:moving-endpoints}
For $0<q<2$ and $s\ge0$,
\[
 \gamma^2\log\E_{\gamma,N\gamma^2/2}[A^{q/\gamma^2}]
 \longrightarrow\Lambda_W(q),\qquad
 \gamma^2\log\E_{\gamma,N\gamma^2/2}[e^{-sA/\gamma^2}]
 \longrightarrow\Psi_W(s).
\]
\end{lemma}

\begin{proof}
The exact recursions depend on $N$ only through the number of steps
and the displayed block biases. Their terminal time is
$N\theta/\pi\to W/2$. Choose a fixed terminal-time upper bound
strictly between $W/2$ and $1$.

For the moment comparison, all admissible arguments remain in a fixed
compact subinterval of $(-1,1)$, with the moment endpoint near $q/4$.
The weight bounds, shifted test functions, and boundary estimates in
Section~\ref{sec:convergence} therefore apply uniformly. The same
induction as in Proposition~\ref{upper:prop:hj-convergence} evaluates
the test function at $(N\theta/\pi,\gamma^2 n/4)$, which converges to
$(W/2,q/4)$. Stirling's formula, the normalization Riemann sum, and
the interpolation to real powers are unchanged.

For the Laplace upper bound, the compact-interval comparison and
exit estimates are uniform for terminal times below the chosen bound.
For the lower bound, choose the characteristic ending at
$(N\theta/\pi,0)$. Smooth dependence on this endpoint puts these
characteristics in a common compact spatial interval, with a common
strict margin below $1$ for the argument of $H$. Thus the clock
moment bounds, variance estimates, and likelihood Riemann sums in
Section~\ref{lower:sec:convergence} apply uniformly. They yield
$4\mathcal G(N\theta/\pi,0)\to4\mathcal G(W/2,0)$.
The case $s=0$ is immediate.
\end{proof}

\subsection{Boundary moments and welding}

Write
$B_{\gamma,w}(p)=\E_{\gamma,w}[R^p]$, allowing the value $+\infty$.
Let $J_{\gamma,w}(x,y)$ be the total mass of
$\mathcal M(w;x,y)$. Left-right symmetry and quantum scaling give
\[
 J_{\gamma,w}(x,y)=J_{\gamma,w}(y,x),\qquad
 J_{\gamma,w}(cx,cy)
 =c^{-1-2w/\gamma^2}J_{\gamma,w}(x,y).
\]
The symmetry follows from reflection of the strip construction of a
thick disk and, for thin disks, reflection of each component of the
Poissonian chain. The scaling follows from
\cite[Lemma~2.19]{ConfWeldDisks}, after disintegration over both lengths.

\begin{lemma}[Boundary identities]\label{lem:boundary-identities}
The boundary moments satisfy
\begin{equation}\label{eq:boundary-reflection}
 B_{\gamma,w}(p)=B_{\gamma,w}(2w/\gamma^2-1-p).
\end{equation}
If $w,r>0$ and $w+r<2$, then
\begin{equation}\label{eq:boundary-welding-moments}
 B_{\gamma,w+r}(p)
 =\frac{B_{\gamma,w}(p-2r/\gamma^2)}
        {B_{\gamma,w}(-2r/\gamma^2)}B_{\gamma,r}(p).
\end{equation}
The denominator in \eqref{eq:boundary-welding-moments} is finite
and positive.
\end{lemma}

\begin{proof}
Substitute $y=1/x$ in the integral of $y^pJ_{\gamma,w}(1,y)$
and use symmetry and homogeneity. Dividing by its integral at $p=0$
gives \eqref{eq:boundary-reflection}, including infinite moments.
In particular, $B_{\gamma,w}(0)=B_{\gamma,w}(2w/\gamma^2-1)=1$.

For the second identity, weld a weight-$w$ disk on the left of a
weight-$r$ disk and integrate over the final right boundary length.
If the interface length is $R$, the mass of the right disk is a
constant times $R^{-2r/\gamma^2}$, by
\cite[Lemmas~2.16 and~2.18]{ConfWeldDisks}. Hence the left disk has
law $\mathbb{P}_{\gamma,w}$ reweighted by $R^{-2r/\gamma^2}$.
The normalization is finite and positive because the welded disk
has finite positive mass with left boundary length one.
After rescaling the right disk to left boundary length one, its law
is $\mathbb{P}_{\gamma,r}$, independently of the left disk. The final right
boundary length is the product of the two boundary ratios.
Taking its $p$th moment proves \eqref{eq:boundary-welding-moments}
by Tonelli's theorem.
\end{proof}

Fix a compact interval $K\subset(0,2)$. All uniform bounds below are
for $w\in K$ and sufficiently small $\gamma$.

\begin{lemma}[Uniform boundary moments]\label{lem:uniform-boundary-moments}
For every fixed $p>0$, there is a finite constant $C_{K,p}$ such that
\[
 \E_{\gamma,w}[R^p]+\E_{\gamma,w}[R^{-p}]\le C_{K,p}.
\]
\end{lemma}

\begin{proof}
For an integer-block weight $w=n\gamma^2/2$, independence of the
boundary ratios and the beta integral give
\begin{equation}\label{eq:integer-boundary-moments}
 B_{\gamma,n\gamma^2/2}(p)
 =\prod_{i=0}^{n-1}
 \frac{Z(-i\theta/\pi+p\theta/\pi)}{Z(-i\theta/\pi)}.
\end{equation}
For fixed $p$, all arguments remain in a compact subset of $(-1,1)$.
The logarithm of each factor is $O(\gamma^2)$, uniformly for weights
in a slightly larger compact interval. There are $O(\gamma^{-2})$
factors, so the product and its reciprocal are uniformly bounded.

For a noninteger weight $w$, choose $r\in(0,\gamma^2/2)$ such that
$w+r=\lceil2w/\gamma^2\rceil\gamma^2/2$, and set
$b=2r/\gamma^2\in(0,1)$. Choose a fixed integer $k\ge\max\{1,p\}$.
By \eqref{eq:boundary-reflection}, $B_{\gamma,r}(b-1)=1$.
Monotonicity of $L^p$ norms applied to $R^{-1}$ gives
$B_{\gamma,r}(-k)\ge1$. Thus
\eqref{eq:boundary-welding-moments} and the integer-block estimate imply
\[
 B_{\gamma,w}(-k-b)
 \le C B_{\gamma,w}(-b)<\infty.
\]
Here finiteness also follows from the same identity: its left-hand
side at weight $w+r$ is finite and both factors on its right are
strictly positive.

Convexity of $\log B_{\gamma,w}$ and $B_{\gamma,w}(0)=1$ give
\[
 B_{\gamma,w}(-b)
 \le B_{\gamma,w}(-k-b)^{b/(k+b)}.
\]
Combining the last two inequalities yields
$B_{\gamma,w}(-k-b)\le C^{(k+b)/k}$, and interpolation with the
zeroth moment bounds $B_{\gamma,w}(-p)$ uniformly.
Finally, for small $\gamma$,
$0<p<2w/\gamma^2-1$ uniformly on $K$.
Equation~\eqref{eq:boundary-reflection} and H\"older's inequality give
$B_{\gamma,w}(p)\le1$. This proves the claim.
\end{proof}

For $0\le b\le1$, define the boundary-tilted law by
\[
 \frac{\dd\widehat{\mathbb{P}}_{\gamma,w,b}}{\dd\mathbb{P}_{\gamma,w}}
 =D_{\gamma,w,b}:=\frac{R^{-b}}{B_{\gamma,w}(-b)}.
\]

\begin{lemma}[Removing the boundary tilt]\label{lem:remove-boundary-tilt}
For each $h>1$ there is a uniform constant $C\ge1$ such that every
nonnegative measurable $Y$ satisfies
\begin{equation}\label{eq:boundary-holder}
 C^{-1}(\E_{\gamma,w}[Y^{1/h}])^h
 \le\widehat\E_{\gamma,w,b}[Y]
 \le C(\E_{\gamma,w}[Y^h])^{1/h}.
\end{equation}
The same inequalities hold with the two laws interchanged.
\end{lemma}

\begin{proof}
Lemma~\ref{lem:uniform-boundary-moments} and interpolation bound all
fixed moments of $R^b$ and $R^{-b}$ uniformly in $b\in[0,1]$.
Moreover, $B_{\gamma,w}(-b)\ge1$: the convex function
$\log B_{\gamma,w}$ vanishes at $0$ and $2w/\gamma^2-1>0$, so it
is nonnegative to the left of zero. Its upper bound follows from
the same moment estimates. Consequently, for each fixed $a>1$,
\[
 \E_{\gamma,w}[D_{\gamma,w,b}^{a}]\le C_a,
 \qquad
 \widehat\E_{\gamma,w,b}[D_{\gamma,w,b}^{-a}]
 =B_{\gamma,w}(-b)^{a-1}B_{\gamma,w}(b(a-1))\le C_a.
\]
H\"older's inequality with exponents $h$ and $h/(h-1)$ gives the
upper comparison and its version with the laws interchanged.
Applying the latter to $Y^{1/h}$ and raising to the power $h$
gives the lower comparison. Truncation covers infinite expectations.
\end{proof}

\subsection{The \texorpdfstring{$N$- and $(N+1)$-block}{N- and (N+1)-block} comparison}

Fix $W\in(0,2)$ and put
\[
 N=\left\lfloor\frac{2W}{\gamma^2}\right\rfloor,\qquad
 W_-=\frac{N\gamma^2}{2},\qquad
 W_+=\frac{(N+1)\gamma^2}{2}.
\]
If $W=W_-$, the ordinary block representation applies directly.
Otherwise set $b_-=2(W-W_-)/\gamma^2$ and
$b_+=2(W_+-W)/\gamma^2$, both in $(0,1)$.
All three weights lie in a fixed compact subset of $(0,2)$ for small
$\gamma$. The hats below specify the boundary-tilted laws just defined.

\begin{lemma}[Area sandwich]\label{lem:area-sandwich}
There are welding couplings such that
\[
 \widehat A_{\gamma,W_-,b_-}\le A_{\gamma,W},
 \qquad
 \widehat A_{\gamma,W,b_+}\le A_{\gamma,W_+}.
\]
\end{lemma}

\begin{proof}
Weld a weight-$(W-W_-)$ disk to the right of a weight-$W_-$ disk,
with the outer left boundary fixed at one. As in
Lemma~\ref{lem:boundary-identities}, the left component has law
$\widehat{\mathbb{P}}_{\gamma,W_-,b_-}$, whereas the complete disk has law
$\mathbb{P}_{\gamma,W}$. Its area is the sum of the two nonnegative
component areas, proving the first inequality. Welding a
weight-$(W_+-W)$ disk to the right of a weight-$W$ disk proves the
second inequality. No estimate on the residual areas is needed.
\end{proof}

\begin{proposition}[Unrestricted transform limits]\label{prop:all-gamma}
The limits in Theorem~\ref{thm:intro-transforms} hold as
$\gamma\downarrow0$ through arbitrary values, with the same
$\Lambda_W$ and $\Psi_W$.
\end{proposition}

\begin{proof}
Fix $h>1$. For $0<q<2$ choose $h$ sufficiently close to one that
$hq<2$. The area sandwich and
Lemma~\ref{lem:remove-boundary-tilt} give
\[
 C^{-1}\bigl(\E_{\gamma,W_-}[A^{q/(h\gamma^2)}]\bigr)^h
 \le\E_{\gamma,W}[A^{q/\gamma^2}]
 \le C\bigl(\E_{\gamma,W_+}[A^{hq/\gamma^2}]\bigr)^{1/h}.
\]
For the upper inequality, first compare the unweighted weight-$W$
law to its boundary tilt, and then use the second welding coupling.
The constants are independent of $\gamma$. Since $W_\pm\to W$,
Lemma~\ref{lem:moving-endpoints} yields
\[
 h\Lambda_W(q/h)
 \le\liminf_{\gamma\downarrow0}\Lambda_{\gamma,W}(q)
 \le\limsup_{\gamma\downarrow0}\Lambda_{\gamma,W}(q)
 \le h^{-1}\Lambda_W(hq).
\]
Continuity and $h\downarrow1$ prove moment-transform convergence.

For the Laplace transform, the function $a\mapsto e^{-sa/\gamma^2}$
is decreasing. The same two comparisons therefore give
\[
 C^{-1}\bigl(\E_{\gamma,W_+}[e^{-sA/(h\gamma^2)}]\bigr)^h
 \le\E_{\gamma,W}[e^{-sA/\gamma^2}]
 \le C\bigl(\E_{\gamma,W_-}[e^{-hsA/\gamma^2}]\bigr)^{1/h}.
\]
For $s>0$, it follows that
\[
 h\Psi_W(s/h)
 \le\liminf_{\gamma\downarrow0}\Psi_{\gamma,W}(s)
 \le\limsup_{\gamma\downarrow0}\Psi_{\gamma,W}(s)
 \le h^{-1}\Psi_W(hs).
\]
Letting $h\downarrow1$ proves the claim; $s=0$ is immediate.
The values of $\gamma$ for which $W=W_-$ are already covered by
Lemma~\ref{lem:moving-endpoints}.
\end{proof}

\section{From transform convergence to the LDP}
\label{sec:joined-rates}\label{sec:full-ldp}

The remaining probability argument is the exponential-tilting
argument underlying the G\"artner--Ellis theorem
\cite[Section~2.3]{DemboZeitouni1998}. It is the same for both
transforms, so we record it once. The model-specific inputs are differentiability,
the ranges of the transform derivatives, and exponential tightness.
No boundary derivative of $I_W$ is needed.
By Proposition~\ref{prop:all-gamma}, all limits in this section are
unrestricted limits as $\gamma\downarrow0$.

For this paragraph, let $Y_\gamma$ be real random variables, let $J$ be
an open interval, and suppose
$L_\gamma(\lambda)=\gamma^2\log\E[e^{\lambda Y_\gamma/\gamma^2}]$
converges to a finite differentiable function $L$ on $J$.
For fixed $\lambda\in J$, define
\[
 \frac{\dd\mathbb{P}_{\gamma,\lambda}}{\dd\mathbb{P}}
 =\exp\!\left\{\frac{\lambda Y_\gamma-L_\gamma(\lambda)}{\gamma^2}\right\},
 \qquad y=L'(\lambda).
\]

\begin{lemma}[Exponential tilting]\label{lem:exponential-tilting}
Under $\mathbb{P}_{\gamma,\lambda}$, $Y_\gamma\to y$ in probability, and
\[
 \lim_{\delta\downarrow0}\liminf_{\gamma\downarrow0}
 \gamma^2\log\mathbb{P}(|Y_\gamma-y|<\delta)
 \ge L(\lambda)-\lambda y.
\]
\end{lemma}

\begin{proof}
For small $h>0$ with $\lambda\pm h\in J$, exponential Markov bounds give
\[
 \limsup_{\gamma\downarrow0}\gamma^2
 \log\mathbb{P}_{\gamma,\lambda}(\pm(Y_\gamma-y)\ge\delta)
 \le L(\lambda\pm h)-L(\lambda)\mp hy-h\delta.
\]
Differentiability makes both bounds negative for sufficiently small $h$.
On $|Y_\gamma-y|<\delta$, the inverse density is at least
$\exp\{[L_\gamma(\lambda)-\lambda y-|\lambda|\delta]/\gamma^2\}$.
Integration over this event and then $\delta\downarrow0$ prove the claim.
\end{proof}

\begin{proof}[Proof of Theorem~\ref{thm:area-ldp}]
Proposition~\ref{prop:all-gamma} provides the unrestricted transform limits.

Apply Lemma~\ref{lem:exponential-tilting} to $Y_\gamma=\log A_\gamma$
with transform $\Lambda_{\gamma,W}$, and to $Y_\gamma=-A_\gamma$
with transform $\Psi_{\gamma,W}$.
The slope ranges proved in Lemmas~\ref{upper:lem:slopes}
and~\ref{lower:lem:slopes} ensure that every $a>a_*(W)$ is reached by
a power tilt and every $0<a<a_*(W)$ by a negative exponential tilt.
Convexity identifies the corresponding costs as the two suprema
defining $I_W$. The lemma therefore gives the required
fixed-neighborhood lower bound at every such $a$.

Markov's inequality gives the matching one-sided upper bounds
\begin{align*}
 \limsup_{\gamma\downarrow0}\gamma^2\log\mathbb{P}(A_\gamma\ge a)
 &\le-I_W(a),\qquad a>a_*(W),\\
 \limsup_{\gamma\downarrow0}\gamma^2\log\mathbb{P}(A_\gamma\le a)
 &\le-I_W(a),\qquad 0<a<a_*(W).
\end{align*}
The endpoint slopes at zero show that these costs are positive.
Hence every fixed neighborhood of $a_*(W)$ has probability tending to
one, which supplies the lower bound there as well.

The unique maximizing parameters tend to zero as $a\to a_*(W)$.
The transform formulas imply that $I_W$ is continuous on $(0,\infty)$
and vanishes at $a_*(W)$.
It is strictly increasing above $a_*(W)$ and strictly decreasing below
it: insert the positive maximizer for one value of $a$ into the
supremum for a more extreme value. This uses no derivative of $I_W$.

For completeness, the endpoint bounds needed for goodness are short.
Any fixed $q\in(0,2)$ gives
$I_W(a)\ge q\log a-\Lambda_W(q)\to\infty$ as $a\to\infty$.
The scalar Laplace formula gives
\[
 -\Psi_W'(s)
 \ge\frac{\sin(\pi W/2)}{\sqrt{\pi s}}-\frac Ws,
\]
because $s=(4/\pi)(C^2-1)H(v)^2$ and $C/\sqrt{C^2-1}\ge1$.
Integration shows $\Psi_W(s)\to-\infty$.
Now $I_W(a)\ge-sa-\Psi_W(s)$, first with $a\downarrow0$ and then
$s\to\infty$, implies $I_W(a)\to\infty$ at zero.
Thus the stated extension of $I_W$ to $\mathbb R$ is a good rate function.

For each positive $a$, the Markov bounds above give a neighborhood
upper bound by using the threshold $a-\delta$ above $a_*(W)$ and
$a+\delta$ below it. At $a_*(W)$ use the bound by one; at zero use
$I_W(\delta)\to\infty$; negative points have neighborhoods of
probability zero. A finite cover yields the upper bound on compact
sets. Finally,
\[
 \limsup_{\gamma\downarrow0}\gamma^2\log\mathbb{P}(A_\gamma>R)
 \le\Lambda_W(q)-q\log R\longrightarrow-\infty
\]
for any fixed $q\in(0,2)$, proving exponential tightness.
This extends the compact upper bound to every closed set.
For an open set, choose a positive point and a neighborhood contained
in the set and apply the lower bounds already proved.
This establishes the LDP and convergence in probability.
\end{proof}

In particular, continuity and monotonicity of $I_W$ give
\begin{equation}\label{eq:lower-tail-probability}
 \lim_{\gamma\downarrow0}\gamma^2\log\mathbb{P}(A_\gamma\le a)=-I_W(a)
 \quad(0<a<a_*(W)),
\end{equation}
and the analogous upper-tail identity for $a>a_*(W)$.
These are consequences of the LDP, not separate probability theorems.

\subsection{The elementary case \texorpdfstring{$W=1$}{W=1}}

For the moment transform, the scalar parameter is $C=1-q$ and
$z=q(2-q)/4$. For the Laplace transform, $v=1/4$ and
$C=\sqrt{1+s/\pi}$. Hence
\[
 \Lambda_1'(q)=-\log[\pi(2-q)],\qquad
 \Psi_1'(s)=\frac{1-\sqrt{1+s/\pi}}s.
\]
Integrating from zero and taking the two transforms yields
\[
 I_1(a)=2\log(2\pi a)-2+\frac1{\pi a},\qquad a>0.
\]
No finite-$\gamma$ inverse-gamma law is asserted.

\bibliographystyle{plain}
\bibliography{refs}
\end{document}